\documentclass[12pt]{amsart}

\pdfoutput=1

\usepackage{amsmath, amssymb}
\usepackage[top=30truemm, bottom=30truemm, left=23truemm, right=23truemm]{geometry}
\usepackage{bm}
\usepackage{mathrsfs}
\usepackage{enumitem} 
\usepackage{graphicx}
\usepackage{float}
\usepackage{mathtools}
\usepackage{tikz}
\usetikzlibrary{arrows.meta} 
\usepackage{bbm}
\usepackage{color}

\theoremstyle{definition}
\newtheorem{definition}{Definition}[section]

\theoremstyle{theorem}
\newtheorem{theorem}[definition]{Theorem}
\newtheorem{lemma}[definition]{Lemma}
\newtheorem{corollary}[definition]{Corollary}
\newtheorem{proposition}[definition]{Proposition}

\newcommand{\AND}{\quad \textrm{and} \quad}

\DeclareMathOperator{\meas}{meas}
\DeclareMathOperator{\supp}{supp}
\DeclareMathOperator*{\Res}{Res}
\DeclareMathOperator{\dist}{dist}

\DeclareMathOperator{\Log}{Log}

\renewcommand{\phi}{\varphi}
\renewcommand{\Re}{\mathrm{Re}\hspace{1pt}}
\renewcommand{\Im}{\mathrm{Im}\hspace{1pt}}
\renewcommand{\epsilon}{\varepsilon}
\renewcommand{\mid}{~;~}

\begin{document}



\title[Limit theorem for the hybrid joint universality theorem]{Limit theorem for the hybrid joint universality theorem on zeta and $L$-functions}
\author[K.~Endo]{K.~Endo}

\maketitle

\begin{abstract}
In 1979, Gonek presented the hybrid joint universality theorem for Dirichlet $L$-functions and proved the universality theorem for Hurwitz zeta-functions with rational parameter as an application. 
Following the introduction of the hybrid universality theorem, 
several generalizations, refinements, and applications have been developed. 
Despite these advancements, no probabilistic proof based on Bagchi's approach has been formulated due to the complexities of adapting his method to the hybrid joint universality theorem.
In this paper, we prove the limit theorem for the hybrid joint universality theorem. 
\end{abstract}


\section{Introduction}

Gonek \cite{G1979} presented the hybrid joint universality theorem for Dirichlet $L$-functions and proved the universality theorem for Hurwitz zeta-functions with rational parameters as an application. 
The concept of the hybrid (joint) universality combines the Voronin's universality theorem and the Kronecker-Weyl approximation theorem.
Here, we mention a slight generalization of Gonek's theorem, developed by Kaczorowski and Kulas \cite{KK2007}, which is stated as follows:

\begin{theorem}\label{thm:HU}
Let $K$ be a compact subset of $\{ s \in \mathbb{C} \mid 1/2 < \Re (s) < 1 \}$ with connected complement,
$\chi_1,\, \ldots,\, \chi_r$ be pairwise non-equivalent Dirichlet characters, and $p_{k_1}, \, \ldots,\, p_{k_N}$ be mutually distinct prime numbers.
Suppose that $f_j$ is a non-vanishing continuous function on $K$ and analytic in the interior of $K$ for $j= 1,\, \ldots,\, r$, and $\theta_{p_{k_1}}, \, \ldots, \, \theta_{p_{k_N}}$ are real numbers.
Then, for any $\epsilon>0$, 
\begin{align*}
\liminf_{T \rightarrow \infty} \frac{1}{T} \meas \Bigg{\{} \tau \in [0, T] ~;~ \max_{1 \leq j \leq n} \max_{s \in K} & \left|  L(s + i \tau, \chi_j) - f_j(s) \right| < \epsilon,\\ 
& \max_{1 \leq n \leq N} \left\| \tau \frac{\log p_{k_n}}{2 \pi} - \theta_{p_{k_n}} \right\| < \epsilon \Bigg{\}} > 0
\end{align*}
holds. 
Here $\meas$ denotes the Lebesgue measure on $\mathbb{R}$, and the symbol $\| x \|$ stands for the distance from $x$ to the nearest integer.
\end{theorem}

Note that several more generalizations and refinements of the hybrid universality theorem have been developed by Pa\' nkowski \cite{P2010, P2013}, and some applications have been developed by Nakamura and Pa\' nkowski \cite{NP2011, NP2012, NP2016}. 

Voronin \cite{V1975} is the first to prove the universality theorem for the Riemann zeta-function. 
After the discovery of Voronin, many generalizations and refinements have been developed. 
We refer the reader to \cite{M2015} for an overview of these developments. 
For the theory and the generalization for other zeta and $L$-function, 
see e.g. \cite{L1996, S2007, K2021, NS2010}. 
One of the most important developments is the probabilistic approach to the universality theorem presented by Bagchi \cite{B1981}. 
Although this approach has often been used to prove the universality, there are no probabilistic proofs for the hybrid universality. 
One of the reasons this approach has not been adopted is that probabilistic proofs are thought to be conceptually more involved (see \cite[the last sentence in p. 221]{KK2007}). 

In this paper, we present the probabilistic proof of the hybrid universality theorem for a wide class of zeta and $L$-functions. 
We adapt the method of Kowalski \cite{K2017, K2021}, 
which is a slightly modified version of Bagchi's method and does not require Prokhorov's theorem and the Birkhoff-Khinchin ergodic theorem in the proof. 
In addition, the denseness lemma for the limit theorem for the hybrid universality in this paper can be proved by a slight modification of the proof of the denseness lemma for the ordinary universality theorem, 
whose statements are given in Lemma \ref{lem:dense} and Proposition \ref{prop:supp}.
Hence, it can be said that the approach in this paper is quite simpler than those previously anticipated. 
Additionally, this method can be said to be useful when combining other types of universality and hybridity. 
Prior to discussing the general case, we see the case of Dirichlet $L$-functions. 

To mention the results, we give some notations. 
Let $\mathbb{T}$ denote the unit circle in the complex plane, that is, $\mathbb{T} = \left\{ z \in \mathbb{C} \mid |z| = 1 \right\}$. 
We put $\mathbb{T}_p = \mathbb{T}$ for any prime number $p$. 
Let $\Omega = \prod_{p} \mathbb{T}_p$, 
where the product runs over all prime numbers $p$. 
For any prime number $p$, 
we denote by $\omega(p)$ the projection of $\omega \in \Omega$ to the coordinate space $\mathbb{T}_p$. 
Since $\mathbb{T}_p$ is a compact Hausdorff topological abelian group, 
there exists the probability Haar measure $\bm{m}_p$ on $\left( \mathbb{T}_p, \mathcal{B} \left( \mathbb{T}_p \right) \right)$. 
Here $\mathcal{B} (S)$ denotes the Borel $\sigma$-field of a topological space $S$.
By the Kolmogorov extension theorem, 
we can construct the probability Haar measure $\bm{m} = \otimes_{p} \bm{m}_p$ on $\left( \Omega, \mathcal{B}\left( \Omega \right) \right)$.
  
For any connected region(open set) $G$ in the complex plane, let $H(G)$ denote the set of holomorphic functions on $G$ equipped with the topology of uniform convergence on compact subsets.
We also define $H_0(G)$ by 
\[
H_0(G)
= \left\{ f \in H(G) \mid \textit{$f(s) \neq 0$ for $s$ $\in$ $G$ or $f \equiv 0$} \right\}.
\]
  
Next we define the random model of the value-distribution for Dirichlet $L$-functions. 
Put 
\[
D = \left\{ s = \sigma + it \mid 1/2 < \sigma < 1 \right\}.
\]
For a Dirichlet character $\chi$, we define the $H(D)$-valued random variable $L(s, \chi, \omega)$ by
\[
L(s, \chi, \omega)
= \prod_{p} \left( 1 - \frac{\chi(p) \omega(p)}{p^{s}} \right)^{-1}. 
\]
Note that it is well-known that $L(s, \chi, \omega)$ converges for $\sigma > 1/2$ almost surely (see e.g. \cite[Lemma 4.2]{S2007}).

In the following, we give the hypothesis of the results for Dirichlet $L$-functions.
Let us fix a $r$-tuple of pairwise non-equivalent Dirichlet characters $\bm{\chi} = (\chi_1,\, \ldots,\, \chi_r)$ and a set of $N$ mutually distinct prime numbers $\mathcal{P}_N = \left\{ p_{k_1}, \, \ldots,\, p_{k_N} \right\}$. 
Denote by $\mathbb{T}^{\mathcal{P}_N}$ the family of elements in $\mathbb{T}$ indexed by $\mathcal{P}_N$, 
that is, $\mathbb{T}^{\mathcal{P}_N} = \left\{ \left( x_p \right)_{p \in \mathcal{P}_N} \mid x_p \in \mathbb{T} \right\}$. 
We define the probability measures $\nu_{T,\, \bm{\chi},\, \mathcal{P}_N }$, $\nu_{\bm{\chi},\, \mathcal{P}_N }$ on $\left(H(D)^r \times \mathbb{T}^{\mathcal{P}_N},\, \mathcal{B}\left(H(D)^r \times \mathbb{T}^{\mathcal{P}_N} \right) \right)$ by
\begin{align*}
&\nu_{T,\, \bm{\chi},\, \mathcal{P}_N }(\bm{A}) 
= \frac{1}{T} \meas\left\{ \tau \in [0, T] \mid \left( \left(L (s + i \tau, \chi_j) \right)_{j = 1}^r,\, (p^{i \tau})_{p \in \mathcal{P}_N} \right) \in \bm{A} \right\}, \\
&\nu_{\bm{\chi},\, \mathcal{P}_N }(\bm{A}) 
= \bm{m}\left( \left\{ \omega \in \Omega \mid \left( \left(L ( s, \chi_j, \omega) \right)_{j = 1}^r,\, (\omega(p))_{p \in \mathcal{P}_N} \right) \in \bm{A} \right\} \right)
\end{align*}
for $\bm{A} \in \mathcal{B}\left( H(D)^r \times \mathbb{T}^{\mathcal{P}_N} \right)$. 

Now we state the limit theorem for the hybrid universality of Dirichlet $L$-functions.

\begin{theorem}\label{thm:main-chi}
Under the above notations and definitions, 
we have the following results:
\begin{enumerate}[label=\rm{(\roman*)}]
\item The probability measures $\nu_{T,\, \bm{\chi},\, \mathcal{P}_N}$ converge weakly to $\nu_{\bm{\chi},\, \mathcal{P}_N }$ as $T \rightarrow \infty$.
\item The support of $\nu_{\bm{\chi},\, \mathcal{P}_N }$ coincides with $H_0(D)^r \times \mathbb{T}^{\mathcal{P}_N}$.
\end{enumerate}
\end{theorem}

\begin{corollary} Theorem \ref{thm:main-chi} implies Theorem \ref{thm:HU}. 
\end{corollary}

\section{Main result}

In this paper, we show the limit theorem for the joint hybrid universality for a wide class of zeta and $L$-functions which includes Dirichlet $L$-functions. 

Let $\phi(s)$ be a Dirichlet series given by
\[
\phi(s)
= \sum_{n = 1}^\infty \frac{a_{\phi}(n)}{n^s}. 
\]
The Dirichlet series $\phi(s)$ is said to belong to the class $\widetilde{\mathcal{S}}$ if it satisfies the following axioms $(S1)$-$(S4)$:
\begin{enumerate}

\item[(S1)] \textit{Ramanujan hypothesis}. The Dirichlet coefficients $a_{\phi}(n)$ satisfy $a_{\phi}(n) \ll_{\epsilon} n^{\epsilon}$ for every $\epsilon>0$.

\item[(S2)] \textit{Analytic continuation}. There exists $\sigma_{1}(\phi) < 1$ such that such that $\phi(s)$ can be meromorphically continued to the half plane $\sigma > \sigma_{1}(\phi)$ and is holomorphic except for at most a pole at $s = 1$.

\item[(S3)] \textit{Finite order}. For any $\epsilon > 0$ and $\sigma_1(\phi) < \sigma_1 < \sigma_2 $, 
there exists a constant $C = C(\epsilon, \sigma_1, \sigma_2; \phi)$ such that $\phi(\sigma + it) \ll |t|^{C}$ uniformly for $\sigma \leq \sigma \leq \sigma_2$ as $|t| \rightarrow \infty$, 
where the implied constant may depend on $\epsilon, \sigma_1, \sigma_2, \phi$.

\item[(S4)] \textit{Polynomial Euler product}. $\phi(s)$ is expressed by the infinite product
\[
\phi(s)
= \prod_{p} \prod_{j = 1}^{m_{\phi}} \left( 1 - \frac{\alpha_{j,\, \phi} (p)}{p^s} \right)^{-1}
\]
where $m_{\phi}$ is a positive integer, and $\alpha_{j,\, \phi} (p) \in \mathbb{C}$ for $j = 1, \ldots, m$.
\end{enumerate}
Denote $\sigma_\phi$ the infimum of all $\sigma_2(\phi) \geq 1/2$ such that
\[
\frac{1}{2T} \int_{- T}^{T} \left| \phi( \sigma + it ) \right|^2 dt \sim \sum_{n = 1}^\infty \frac{ |a_{\phi}(n)|^2 }{n^{2 \sigma}}
\]
holds for any $\sigma \geq \sigma_{2}(\phi)$. 
We further denote $D_{\phi}$ by $D_{\phi} = \{ s \mid \sigma_{\phi} < \sigma < 1 \}$. 

Note that the class $\widetilde{\mathcal{S}}$ is a slightly wide class of the Steuding class $\mathcal{S}$ (see \cite{S2007}).
If $\phi(s) \in \widetilde{S}$ satisfies the following condition (S5), 
$\phi(s)$ is called to belong to the Steuding class $\mathcal{S}$:
\begin{itemize}
\item[(S5)] \textit{Prime mean-square}. There exists a positive constant $\kappa$ such that
\[
\lim_{x \rightarrow \infty} \frac{1}{\pi(x)} \sum_{p \leq x} \left| a_\phi (p) \right|^2 = \kappa.
\]
\end{itemize}

Next, we define a random model for the value distribution of an element in $\widetilde{\mathcal{S}}$. 
For $\phi \in \widetilde{\mathcal{S}}$ and $\omega \in \Omega$, 
we define $\phi(s,\, \omega)$ by
\[
\phi(s, \omega) 
= \prod_{p} \prod_{j = 1}^{m_{\phi}} \left( 1 - \frac{\alpha_{j,\, \phi}(p) \omega(p) }{p^s} \right)^{-1}, 
\]
which is known to converge for $\sigma > 1/2$ almost everywhere (see \cite[p. 65]{S2007}). 
Hence we treat $\phi(s,\, \omega)$ as an $H(D_{\phi})$-valued random variable. 

To state the main theorem, let us fix $\phi_1, \ldots, \phi_r \in \widetilde{\mathcal{S}}$ and a set of $N$ mutually distinct prime numbers $\mathcal{P}_N = \left\{ p_{k_1}, \, \ldots,\, p_{k_N} \right\}$.
We use the notations
\[
\underline{\phi}(s) = \left( \phi_1(s), \ldots, \phi_r(s) \right) \AND 
\underline{\phi}(s, \omega) = \left( \phi_1(s, \omega), \ldots, \phi_r(s, \omega) \right) .
\]
Note that $\underline{\phi}(s, \omega)$ can be treated as a $\prod_{j = 1}^r H(D_{\phi_j})$-valued random variable. 
We define the probability measures $\nu_{T, \underline{\phi},  \mathcal{P}_N }$, $\nu_{\underline{\phi},\, \mathcal{P}_N }$ on $\left( \prod_{j = 1}^r H(D_{\phi_j}) \times \mathbb{T}^{\mathcal{P}_N}, \, \mathcal{B}\left( \prod_{j = 1}^r H(D_{\phi_j}) \times \mathbb{T}^{\mathcal{P}_N} \right) \right)$ by
\begin{align*}
&\nu_{T, \underline{\phi},  \mathcal{P}_N }\left( \bm{A} \right)
= \frac{1}{T} \meas\left\{ \tau \in [0, T] \mid \left( \underline{\phi}( s+ i \tau),\, (p^{i \tau})_{p \in \mathcal{P}_N} \right) \in \bm{A} \right\}, \\
&\nu_{\underline{\phi}, \mathcal{P}_N } \left( \bm{A} \right)
= \bm{m} \left( \left\{ \omega \in \Omega \mid \left( \underline{\phi}(s, \omega),\, (\omega(p))_{p \in \mathcal{P}_N} \right) \in \bm{A}  \right\}\right). 
\end{align*}

Then the main theorem is as follows:

\begin{theorem}\label{thm:main}
Let $\phi_1, \ldots, \phi_r \in \widetilde{\mathcal{S}}$ and let $\mathcal{P}_N = \left\{ p_{k_1}, \, \ldots,\, p_{k_N} \right\}$ be a set of $N$ mutually distinct prime numbers. 
Assume that the set of all convergent elements
\begin{equation}\label{eqn:DSC}
\left( \sum_{p > X} \frac{a_{\phi_1}(p) c(p)}{p^{s}}, \ldots, \sum_{p > X} \frac{a_{\phi_r}(p) c(p)}{p^{s}} \right),\, c(p) \in \mathbb{T}
\end{equation}
is dense in $\prod_{j = 1}^r H(D_{\phi_j})$ for any $X>0$.
Then we have the following:
\begin{enumerate}[label=\rm{(\roman*)}]
\item The probability measures $\nu_{T, \underline{\phi},  \mathcal{P}_N }$ converge weakly to $\nu_{\underline{\phi},\, \mathcal{P}_N }$ as $T \rightarrow \infty$.
\item The support of $\nu_{\underline{\phi},\, \mathcal{P}_N }$ coincides with $\prod_{j = 1}^r H_0(D_{\phi_j}) \times \mathbb{T}^{\mathcal{P}_N}$.
\end{enumerate}
\end{theorem}

\begin{corollary}\label{cor:main}
Let every settings and assumptions be the same as in Theorem \ref{thm:main}.
Let $K_j$ be a compact subset of $D_{\phi_j}$ with connected complement for $j = 1, \ldots, r$. 
Suppose that $f_j$ is a non-vanishing continuous function on $K_j$ and analytic in the interior of $K_j$ for $j= 1,\, \ldots,\, r$, and that $\theta_{p_{k_1}}, \, \ldots, \, \theta_{p_{k_N}}$ are real numbers.
Then, for any $\epsilon>0$, we have
\begin{align*}
\liminf_{T \rightarrow \infty} \frac{1}{T} \meas \Bigg{\{} \tau \in [0, T] ~;~ \max_{1 \leq j \leq r} \max_{s \in K_j} & \left|  \phi_j(s + i \tau) - f_j(s) \right| < \epsilon,\\ 
& \max_{1 \leq n \leq N} \left\| \tau \frac{\log p_{k_n}}{2 \pi} - \theta_{p_{k_n}} \right\| < \epsilon \Bigg{\}} > 0.
\end{align*}
\end{corollary}

Note that, it is proved by Bagchi \cite{B1981} that, for pairwise non-equivalent Dirichlet characters $\chi_1,\, \ldots,\, \chi_r$, 
the set of all all convergent elements
\[
\left( \sum_{p > X} \frac{\chi_1(p) c(p)}{p^{s}}, \ldots, \sum_{p > X} \frac{\chi_r(p) c(p)}{p^{s}} \right),\, c(p) \in \mathbb{T}
\]
is dense in $H(D)^r$.
Hence, Theorem \ref{thm:main-chi} is a special case of Theorem \ref{thm:HU}.

Finally, we give one more corollary. 
Let $\phi (s)$ be an element of the Steuding class $\mathcal{S}$. 
The assumption of the type of the axiom $(S5)$ was firstly introduced by Laurin\u{c}ikas and Matsumoto \cite{LM2000}. 
This assumption implies that the set of all convergent series
\[
\sum_{p>X} \frac{a_{\phi} (p) c(p)}{p^s}, \quad c(p) \in \mathbb{T}
\]
is dense in $H \left( D_{\phi} \right)$, which is proved in \cite[Theorem 5.10]{S2007}.
Thus, we obtain the following corollary:

\begin{corollary}
Let $\phi (s)$ be an element of the Steuding class $\mathcal{S}$ and let $K$ be a compact subset of $D_\phi$. 
Suppose that $f$ is a non-vanishing continuous function on $K$ and analytic in the interior of $K$ and that $\theta_{p_{k_1}}, \, \ldots, \, \theta_{p_{k_N}}$ are real numbers.
Then, for any $\epsilon>0$, we have
\[
\liminf_{T \rightarrow \infty} \frac{1}{T} \meas\left\{ \tau \in [0, T] \mid \max_{s \in K} \left| \phi(s + i \tau) - f (s) \right| < \epsilon,\, \max_{1 \leq n \leq N} \left\| \tau \frac{\log p_{k_n}}{2 \pi} - \theta_{p_{k_n}} \right\| < \epsilon \right\}>0. 
\]
\end{corollary}

\section{Preliminaries}

Fix $\phi \in \widetilde{\mathcal{S}}$. 
We define a metric $d_\phi$ on $H(D_{\phi})$ as follows: 
We take a sequence of compact subsets $\left\{ K_{\ell,\, \phi} \right\}_{\ell = 1}^\infty$ satisfying
\begin{itemize}

\item $\displaystyle D_{\phi} = \bigcup_{\ell = 1}^\infty K_{\ell,\, \phi}$, 

\item $K_{\ell,\, \phi} \subset K_{\ell + 1,\, \phi}^{\circ}$ for any $\ell \in \mathbb{N}$, 
where $A^\circ$ denotes the interior of a set $A$, and

\item for any compact subset $K$ of $\displaystyle D_{\phi}$, 
there exists $\ell \in \mathbb{N}$ such that $K \subset K_{\ell, \phi}$.

\end{itemize}

We define the metrics (semi-norms) $\left\{d_{\ell, \phi}\right\}_{\ell = 1}^\infty$ on $H(D_\phi)$ by 
\[
d_{\ell, \phi} (f, g) = \sup_{s \in K_{\ell, \phi}}\left| f(s) - g(s) \right|
\]
for $f, g \in H(D_\phi)$.
For the collection of the metrics $\left\{ d_{\ell, \phi} \right \}_{\ell = 1}^\infty$, 
we define the metric $d_{\phi}$ on $H(D_\phi)$ by
\[
d_{\phi} (f, g)
= \sum_{\ell = 1}^\infty \frac{ d_{\ell, \phi} (f, g) \land 1  }{2^{\ell}}
\]
for $f, g \in H(D_\phi)$,
which induces the desired topology on $H(D_\phi)$.
Here $a \land b $ denotes the minimum of real numbers $a$ and $b$.

Next, we put $\mathbb{T}_p = \mathbb{T}$ for prime numbers and see the set $\mathbb{T}_p$ as a metric space as follows. 
Recall that $\mathbb{R}/\mathbb{Z}$ is a metric space with respect to the metric $d_{\mathbb{R}/\mathbb{Z}}$ defined by 
\[
d_{\mathbb{R}/\mathbb{Z}} \left( \theta_1 + \mathbb{Z},\, \theta_2 + \mathbb{Z} \right)
= \| \theta_1 - \theta_2 \|
\]
for $\theta_1 + \mathbb{Z},\, \theta_2 + \mathbb{Z} \in \mathbb{R}/\mathbb{Z}$. 
It is known that the mapping $T_p:\mathbb{R}/\mathbb{Z} \ni \theta_p + \mathbb{Z} \mapsto \exp(2 \pi i \theta_p) \in \mathbb{T}_p $ is homeomorphic. 
Hence we can see $\mathbb{T}_p$ as a metric space, whose metric $d_p$ is defined by
\[
d_p(x_p, y_p) 
= d_{\mathbb{R}/\mathbb{Z}} \left( T_p^{-1}(x_p), T_p^{-1}(y_p) \right)
\]
for $x_p,\, y_p \in \mathbb{T}_p$.

Next, let us fix $\phi_1, \ldots, \phi_r \in \widetilde{\mathcal{S}}$ and let $\mathcal{P}_N = \left\{ p_{k_1}, \, \ldots,\, p_{k_N} \right\}$ be a set of $N$ mutually distinct prime numbers.
Then we can define the metric $d$ on $\prod_{j = 1}^r H(D_{\phi_j}) \times \mathbb{T}^{\mathcal{P}_N}$ by
\[
d \left( \left( \bm{f}, \bm{x} \right), \left(\bm{g}, \bm{y} \right) \right)
= \sum_{j = 1}^r d_{\phi_j} \left( f_j,\, g_j \right)\
+ \sum_{n = 1}^N d_{p_{k_n}} (x_{p_{k_n}} ,\, y_{p_{k_n}} )
\]
for any $(\bm{f}, \bm{x}) = \left(  (f_j)_{j = 1}^r, ( x_{p_{k_n}} )_{n = 1}^N \right),\, (\bm{g}, \bm{y}) = \left(  (g_j)_{j = 1}^r, ( y_{p_{k_n}} )_{n = 1}^N \right) \in \prod_{j = 1}^{r} H(D_{\phi_j}) \times \mathbb{T}^{\mathcal{P}_N}$.

At the end of this section, 
we give some notations. 
For a probability space $( M, \mathcal{M}, \mathbb{P} )$ and a random variable $X$ that values a measurable space $(S, \mathcal{S})$, we use the notation $\mathbb{P} \left( X \in A \right)
= \mathbb{P} \left( \left\{ \omega \in M \mid X(\omega) \in A \right\} \right)$
for $A \in \mathcal{S}$, 
and the expectation of a complex-valued random variable $Y$ is denoted by $\mathbb{E}^{\mathbb{P}} \left[ Y \right]$.

\section{Proof of Theorem \ref{thm:main} \textrm{(i)}}

To use the smoothing technique in analytic number theory, 
let us fix a real-valued smooth function $\lambda(x)$ on $[0, \infty)$ with compact support satisfying $\lambda(x)=1$ on $[0,1]$ and $0 \leq \lambda(x) \leq 1$ for all $ x \in [0, \infty)$. 
Let $\widehat{\lambda}(s)$ denote the Mellin transform of $\lambda(x)$, that is, 
\[
\widehat{\lambda}(s)
=  \int_{0}^{\infty} \lambda(x) x^{s - 1} dx
\]
for $\Re(s)> 0$. 
We recall the basic properties of the Mellin transform of $\lambda(x)$.

\begin{lemma}\label{lem:MeTr}
We have the following:
\begin{enumerate}[label=\rm{(\roman*)}]

\item The Mellin transform $\widehat{\lambda}(s)$ can be meromorphically continued to $\Re(s) > -1$ and has only one simple pole at $s = 0$ with residue $\lambda(0) = 1$.

\item Let $-1 < A < B$ be real numbers and let $N$ be a positive integer. 
Then there exists $C(A, B; N)>0$ such that
\[
\left| \widehat{\lambda}(s) \right|
\leq C(A, B; N) \left( 1 + | t | \right)^{-N}
\]
for $s = \sigma + it$ with $A \leq \sigma \leq B$.

\item For any $c > 0$ and $x>0$, the Mellin inversion formula
\[
\lambda(x) = \frac{1}{2 \pi i} \int_{c - i \infty}^{c + i \infty} \widehat{\lambda}(s) x^{-s} ds
\]
holds. 
\end{enumerate}
\end{lemma}

\begin{proof}
The proof can be found in \cite[Appendix A]{K2021}.
\end{proof}

For $X \geq 2$, we define the function $\phi_X(s)$ by
\[
\phi_X(s) 
= \sum_{n = 1}^{\infty} \frac{a_\phi(n) \lambda(n/X) }{n^s}. 
\]

We use the following lemma, which is essentially deduced from the Carlson's theorem (see e.g. \cite[\S\,9.51]{T1986}).

\begin{lemma}\label{lem:UME}
Let $\phi(s)$ be an element of $\widetilde{\mathcal{S}}$.
Suppose that $\sigma_1$ and $\sigma_2$ satisfy $\sigma_\phi < \sigma_1 < \sigma_2 < 1$. 
Then we have
\[
\frac{1}{2T} \int_{- T }^T \left| \phi(\sigma + it) \right|^2 dt 
\sim \sum_{n = 1}^{\infty} \frac{ | a_{\phi}(n) |^2 }{n^{2 \sigma}}
\]
uniformly for $\sigma_1 \leq \sigma \leq \sigma_2$. 
\end{lemma}

Lemma \ref{lem:UME} is only stated as a fact in Steuding's textbook \cite{S2007}. 
In this paper, we provide a full proof. 

\begin{proof}
Let us fix $\sigma_1$ and $\sigma_2$ satisfying $\sigma_\phi < \sigma_1 < \sigma_2 < 1$.
First, we show
\begin{equation}\label{eqn:supES}
\frac{1}{2T} \int_{-T}^{T} \left| \phi_X( \sigma + it ) \right|^2 \, dt
= \sum_{n =1}^{\infty} \frac{ \left| a_\phi (n) \right|^2 }{n^{2\sigma}} + E_1\left( \sigma, X \right) + O \left( X T^{-1} \right)
\end{equation}
uniformly for $\sigma_1 \leq \sigma \leq \sigma_2$, 
where $E_1\left( \sigma, X \right)$ depends only on $\sigma$ and $X$, and is estimated as
\[
\lim_{X \rightarrow \infty} \sup_{\sigma_1 \leq \sigma \leq \sigma_2} E_1\left( \sigma, X \right) = 0.
\]
Using Hilbert's inequality \cite{MV1974}, 
we have
\begin{align*}
\frac{1}{2T} \int_{-T}^{T} \left| \phi_X( \sigma + it ) \right|^2 \, dt 
&= \frac{1}{2T} \int_{-T}^{T} \left| \sum_{n = 1}^{\infty} \frac{ a_\phi (n) \lambda (n/X)}{n^{\sigma + it}} \right|^2 \, dt \\
&= \sum_{n = 1}^{\infty} \frac{ \left| a_\phi (n) \right|^2 \lambda(n/X)^2 }{n^{2 \sigma}}
+O \left( \frac{1}{T} \sum_{n = 1}^{\infty} n \cdot \frac{ \left| a_\phi (n) \right|^2 \lambda(n/X)^2 }{n^{2 \sigma}} \right) \\
&= \sum_{n = 1}^{\infty} \frac{ \left| a_\phi (n) \right|^2}{n^{2 \sigma}} + E_1(\sigma, X) 
+O \left( \frac{1}{T} \sum_{n = 1}^{\infty} n \cdot \frac{ \left| a_\phi (n) \right|^2 \lambda(n/X)^2 }{n^{2 \sigma}} \right),
\end{align*}
where
\[
E_1(\sigma, X)
=  \sum_{n = 1}^{\infty} \frac{ \left| a_\phi (n) \right|^2 (1 - \lambda(n/X)^2) }{n^{2 \sigma}}
\]
Note that, letting a constant positive $C_\lambda$ satisfy $\supp(\lambda) \subset [0, C_\lambda]$, 
we have $\lambda(n/X) =0 $ for $n > C_\lambda X$.
Hence we have
\[
\frac{1}{T} \sum_{n = 1}^{\infty} n \cdot \frac{ \left| a_\phi (n) \right|^2 \lambda(n/X)^2 }{n^{2 \sigma}} 
\leq \frac{C_\lambda X}{T} \sum_{n = 1}^{\infty} \frac{ \left| a_\phi (n) \right|^2 \lambda(n/X)^2 }{n^{2 \sigma}} 
\leq \frac{X}{T} \sum_{n = 1}^{\infty} \frac{ \left| a_\phi (n) \right|^2 }{n^{2 \sigma_1}} 
\ll \frac{X}{T},
\]
where we can find that the series $\sum_{n = 1}^{\infty} \left| a_\phi (n) \right|^2 / n^{2 \sigma_1}$ converges by the axiom (i) in $\widetilde{\mathcal{S}}$. 
Since $\lambda(n/X) = 1$ for $n \leq X$ by the definition of $\lambda$, 
we have
\[
\sup_{\sigma_1 \leq \sigma \leq \sigma_2} E_1(\sigma, X) 
\leq \sum_{n > X} \frac{ \left| a_\phi (n) \right|^2 }{n^{2 \sigma_1}} \rightarrow 0
\]
as $X \rightarrow \infty$ by the convergence of the series $\sum_{n = 1}^{\infty} \left| a_\phi (n) \right|^2 / n^{2 \sigma_1}$.
Therefore we arrive at the estimate \eqref{eqn:supES}. 
In what follows, we take $X = T^{1/2}$.

By the definition of $\sigma_\phi$, 
we can take $\alpha$ such that $\sigma_\phi \leq \alpha < \sigma_1$ and 
\begin{equation}\label{eqn:MTC}
\frac{1}{2T} \int_{- T}^{T} \left| \phi( \alpha + it ) \right|^2 dt \sim \sum_{n = 1}^\infty \frac{ |a_{\phi}(n)|^2 }{n^{2 \alpha}}
\end{equation}
as $T \rightarrow \infty$. 

For $\sigma_1 \leq \Re(z) \leq \sigma_2$ and $c>1$, 
we have
\[
\phi_{X}(z)
= \frac{1}{2 \pi i} \sum_{n = 1}^\infty \int_{c - i \infty}^{c + i \infty} \frac{a_\phi (n) \widehat{\lambda} (\xi) }{n^{z}} \left( \frac{n}{X} \right)^{ - \xi } d \xi
= \frac{1}{2 \pi i} \int_{c - i \infty}^{c + i \infty} \phi(z + \xi) \widehat{\lambda} (\xi) X^{\xi} d \xi
\]
by the Mellin inversion formula for $\lambda(x)$, as is stated in Lemma \ref{lem:MeTr} (iii). 
Note that the interchange of the above sum and the integral is justified by Fubini's theorem and Lemma \ref{lem:MeTr} (ii).
We shift the path of integration to the line $\Re(\xi) = - \Re(z) + \alpha $ (see Figure \ref{fig:path1} below). 

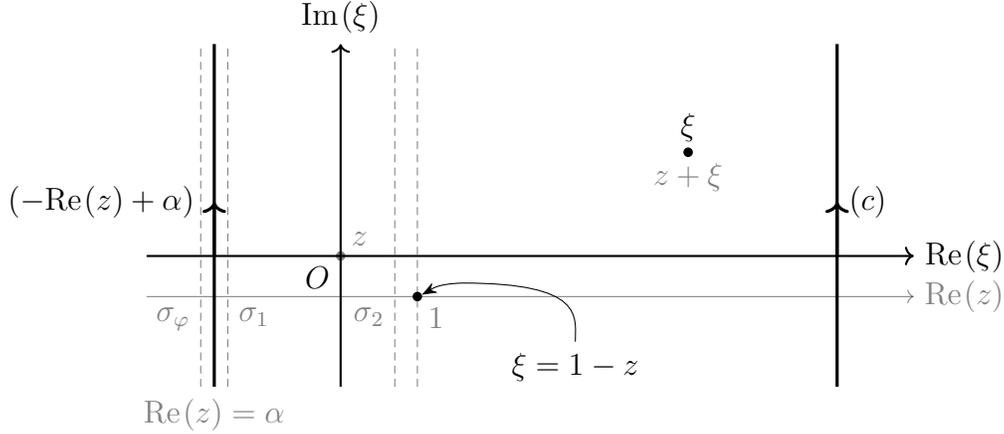
\begin{figure}[H]
\centering
\def\xs{5}
\begin{tikzpicture}[scale = 1.2, xscale = \xs]
\def\RX{2.1}
\def\LX{0.4}
\def\UY{2.8}
\def\DY{-1}
\def\sphi{0.52}
\def\sf{0.58} 
\def\ss{0.95} 
\draw[->, opacity = .5] (\LX, 0) -- (\RX, 0) node[right, opacity = .5]{$\Re(z)$};
\draw[densely dashed, opacity = .5] (\sphi, \DY) -- (\sphi, \UY);
\draw (0.52, 0) node[below left, opacity = .5]{$\sigma_\phi$};
\draw[densely dashed, opacity = .5] (\sf, \DY) -- (\sf, \UY);
\draw (\sf, 0) node[below right, opacity = .5]{$\sigma_1$};
\draw[densely dashed, opacity = .5] (\ss, \DY) -- (\ss, \UY);
\draw (\ss, 0) node[below left, opacity = .5]{$\sigma_2$};
\draw[densely dashed, opacity = .5] (1, \DY) -- (1, \UY);
\draw (1, 0) node[below right, opacity = .5]{$1$}; 
\def\zx{0.83}
\def\zy{0.45}
\begin{scope}[shift = {(\zx, \zy)}, xscale = 1/\xs]
\fill[opacity = .5] (0, 0) circle(1.5pt) node[above right, opacity = .5]{$z$};
\draw (0, 0) node[below left]{$O$}; 
\end{scope}
\draw[thick, ->] (\LX, \zy) -- (\RX, \zy) node[right]{$\Re(\xi)$}; 
\draw[thick, ->] (\zx, \DY) -- (\zx, \UY) node[above]{$\Im(\xi)$}; 
\draw[very thick] (\zx + 1.1, \DY) -- (\zx + 1.1, \UY); 
\draw[very thick] (\sphi/2 + \sf/2, \DY) node[below, opacity = .5]{$\Re(z) = \alpha$} -- (\sphi/2 + \sf/2, \UY); 
\draw[very thick, ->] (\zx + 1.1, \zy) -- (\zx + 1.1, \zy + 0.6) node[right]{$(c)$}; 
\draw[very thick, ->] (\sphi/2 + \sf/2, \zy) -- (\sphi/2 + \sf/2, \zy + 0.6) ;
\draw (\sphi/2 + \sf/2 - 0.02, \zy + 0.6) node[left]{$(- \Re(z) + \alpha)$}; 
\begin{scope}[shift = {(1.6, 1.6)}, xscale = 1/\xs]
\fill (0, 0) circle(1.5pt) node[above]{$\xi$};
\draw (0, 0) node[below, opacity = .5]{$z + \xi$};
\end{scope}
\begin{scope}[shift = {(1, 0)}, xscale = 1/\xs]
\fill (0, 0) circle(1.5pt);
\end{scope}
\draw[-{Stealth[length=2mm]}]
(1.35, -0.5) node[below]{$\xi = 1 - z$} to [out = 90, in = 0]
(1.1, 0.15) to [out = 190, in = 70]
(1.01, 0.02);
\end{tikzpicture}
\caption{The shift of the path of the integration}
\label{fig:path1}
\end{figure}

By Lemma \ref{lem:MeTr} (ii) and the axiom $(S3)$ of $\widetilde{\mathcal{S}}$, 
we can shift the above path of integration to deduce
\begin{equation}\label{eqn:supP1}
\phi(z) - \phi_{X}(z) 
= - \frac{1}{2 \pi i} \int_{- \Re(z) + \alpha - i \infty }^{ - \Re(z) + \alpha + i \infty} \phi(z + \xi) \widehat{\lambda}(\xi) X^{\xi} d \xi - E(z),
\end{equation}
where 
\begin{equation} \label{eqn:RE}
E(z)
= \begin{dcases}
\Res_{\xi = 1 - z} \left( \phi(z + \xi) \widehat{\lambda}(\xi) X^{\xi} \right) & \textrm{if $\phi(s)$ has a pole at $s=1$}, \\
0 & \textrm{otherwise}. 
\end{dcases}
\end{equation}

We estimate $E(z)$ in the case $\phi(s)$ has a pole of order $f$ at $s=1$. 
We may write
\begin{align*}
E(z)
&= \Res_{w = 1 } \left( \phi(w) \widehat{\lambda}(w - z) X^{w - z} \right) \\
&= \lim_{w \rightarrow 1} \frac{1}{(f - 1)!} \left( \frac{d}{dw} \right)^{f - 1} \left( (w - 1)^f \phi(w) \widehat{\lambda}(w - z) X^{w - z} \right) \\
&=: \lim_{w \rightarrow 1} \widetilde{E}(w).
\end{align*}
Let $\epsilon$ be a small positive real parameter. 
Then, by Cauchy's integral formula, we have
\[
\widetilde{E}(w) 
= \frac{1}{2 \pi i} \int_{\left| \eta - w \right| = \epsilon} \frac{ (\eta - 1)^f \phi(\eta) \widehat{\lambda}( \eta - z ) X^{\eta - z} }{ \left( \eta - w \right)^{ f } } d \eta
\]
for $|w| < \epsilon /2 $, where the path of integration is oriented counterclockwise.
Since $\phi(\eta)$ has a pole of order $f$ at $\eta = 1$, $(\eta - 1)^f \phi(\eta)$ is bounded around $\eta = 1$.
Hence we have
\[
\widetilde{E}(w) 
\ll \epsilon^{1-f} X^{1 - \Re(z) + \frac{3}{2}\epsilon} \max_{ \eta \mid \left| \eta - w \right| = \epsilon}\left| \widehat{\lambda}( \eta - z ) \right|.
\]
Therefore, taking $\epsilon$ small enough, we have the bound
\begin{equation}\label{eqn:E-res}
E(z)
\ll \frac{X}{ \left( | \Im (z) | + 2 \right)^2 }. 
\end{equation}
uniformly for $\sigma _1 \leq \Re(z) \leq \sigma_2$ by Lemma \ref{lem:MeTr} (ii).

Put $\delta = \sigma_1 - \alpha$. 
Then we have
\begin{align*}
& \frac{1}{2T} \int_{-T}^T \left| \phi(\sigma + it) - \phi_X(\sigma + it) \right|^2 \,dt \\
\ll&\,  X^{- \delta} \int_{- \infty}^\infty \left| \widehat{\lambda}( - \sigma + \alpha + i v ) \right| \, dv \left( \frac{1}{2T} \int_{-T}^T \left| \phi(\alpha + i(t + v)) \right|^2 dt \right) 
+ \frac{1}{2T} \int_{-T}^T \left| E(\sigma + i t) \right|^2 \, dt \\
=&\, S_1 + S_2.
\end{align*}

We estimate $S_1$. 
Now, since the estimate
\[
\frac{1}{2T} \int_{-T}^T \left| \phi(\alpha + i(t + v)) \right|^2 dt
\leq \frac{1}{2T} \int_{-T - |v|}^{T + |v|} \left| \phi(\alpha + i t) \right|^2 dt
\ll \frac{1}{T} \left( T + |v| \right)
\ll 1 + |v|
\]
holds by the asymptotic formula \eqref{eqn:MTC}, 
we have
\[
S_1 
\ll X^{- \delta} \int_{- \infty}^\infty \left| \widehat{\lambda}( - \sigma + \alpha + i v ) \right| \left( 1 + |v| \right) \, dv 
\ll X^{- \delta}
\]
uniformly for $\sigma_1 \leq \sigma \leq \sigma_2$ by Lemma \ref{lem:MeTr} (ii). 

As for the estimate $S_2$, we have $S_2 \ll X T^{-1}$ uniformly for $\sigma_1 \leq \sigma \leq \sigma_2$ by the estimate \eqref{eqn:E-res}. 
Hence we obtain
\[
\frac{1}{2T} \int_{-T}^T \left| \phi(\sigma + it) - \phi_X(\sigma + it) \right|^2 \,dt
\ll X^{- \delta} + X T^{-1}.
\]
Now, using Minkowski's inequality, we have
\begin{align*}
&\left| \left( \frac{1}{2T} \int_{-T}^{T} \left| \phi(\sigma + it) \right|^2\, dt \right)^{1/2} - \left( \frac{1}{2T} \int_{-T}^{T} \left| \phi_X(\sigma + it) \right|^2\, dt \right)^{1/2} \right| \\
\leq&\, \left( \frac{1}{2T} \int_{-T}^T \left| \phi(\sigma + it) - \phi_X(\sigma + it) \right|^2 \,dt \right)^{1/2}
\ll \left( X^{- \delta} + XT^{-1} \right)^{1/2}.
\end{align*}
We also have
\begin{align*}
& \left( \frac{1}{2T} \int_{-T}^{T} \left| \phi(\sigma + it) \right|^2\, dt \right)^{1/2} \\
\leq& \left( \frac{1}{2T} \int_{-T}^T \left| \phi(\sigma + it) - \phi_X(\sigma + it) \right|^2 \,dt \right)^{1/2}
+ \left( \frac{1}{2T} \int_{-T}^{T} \left| \phi_X(\sigma + it) \right|^2\, dt \right)^{1/2} 
\ll 1.
\end{align*}
Hence, using the inequality 
\[
| x - y | \leq \left| x^{1/2} + y^{1/2} \right| \left| x^{1/2} - y^{1/2} \right|
\]
for $x,\, y \geq 0$ and the estimate \eqref{eqn:supES}, 
we have
\begin{align*}
&\left| \frac{1}{2T} \int_{-T}^{T} \left| \phi(\sigma + it) \right|^2\, dt - \frac{1}{2T} \int_{-T}^{T} \left| \phi_X(\sigma + it) \right|^2\, dt \right| \\
\ll&\, \left| \left( \frac{1}{2T} \int_{-T}^{T} \left| \phi(\sigma + it) \right|^2\, dt \right)^{1/2} - \left( \frac{1}{2T} \int_{-T}^{T} \left| \phi_X(\sigma + it) \right|^2\, dt \right)^{1/2} \right| \\
\ll&\, \left( X^{- \delta} + XT^{-1} \right)^{1/2}.
\end{align*}
Therefore, after short calculations, we obtain
\[
\frac{1}{2T} \int_{- T }^T \left| \phi(\sigma + it) \right|^2 dt 
= \sum_{n = 1}^{\infty} \frac{ | a_{\phi}(n) |^2 }{n^{2 \sigma}} \left( 1 + E \left( \sigma, T \right) \right), 
\]
where the estimate
\[
\lim_{T \rightarrow \infty} \sup_{\sigma_1 \leq \sigma \leq \sigma_2} | E \left( \sigma, T \right) | = 0
\]
holds. 
This completes the proof.
\end{proof}

\begin{proposition}\label{prop:MT}
Let $\phi(s)$ be an element of $\widetilde{\mathcal{S}}$ and $C$ be a compact subset of $D_{\phi}$. 
Then we have
\[
\lim_{X \rightarrow \infty} \limsup_{T \rightarrow \infty} \frac{1}{T} \int_{0}^{T} \max_{s \in C} \left| \phi(s + i \tau) - \phi_X (s + i \tau) \right| d \tau = 0
\]
\end{proposition}

\begin{proof}
The following proof is almost the same as in \cite[Lemma 4.8 in p.~73]{S2007}.
We fix a rectangle $\mathcal{R} = [\sigma_L, \sigma_R] \times i [ - R, R ]$ satisfying
\[
\sigma_\phi < \sigma_L < \min_{s \in C} \Re(s) \leq \max_{s \in C} \Re(s) < \sigma_R < 1
\AND R > |\max_{s \in C} \Im(s)| \lor |\min_{s \in C} \Im(s)|,
\]
where $a \lor b$ stands for the maximum of real numbers $a$ and $b$ (see Figure \ref{fig:rec} below). 

\begin{figure}[H]\label{fig:rec} 
\centering
\def\xs{1}
\begin{tikzpicture}[scale = 16, xscale = \xs]
\def\LX{0.45}
\def\RX{1.05}
\def\UY{0.2}
\def\DY{-0.2}
\def\R{0.17}
\def\sphi{0.52}
\def\sf{0.58} 
\def\ss{0.95} 
\draw[->, thick] (\LX, 0) -- (\RX, 0) node[right]{$\sigma$}; 
\draw (0.5, 0) node[below left]{$1/2$}; 
\draw (1, 0) node[below right]{$1$}; 
\draw[densely dashed, thick] (0.5, \DY) -- (0.5, \UY);
\draw[densely dashed, thick] (1, \DY) -- (1, \UY);
\draw[densely dashed] (\sphi, \DY) -- (\sphi, \UY);
\draw (0.52, 0) node[below right]{$\sigma_\phi$};
\draw[densely dashed] (\sf, \DY) -- (\sf, \UY);
\draw (\sf, 0) node[below right]{$\sigma_L$};
\draw[densely dashed] (\ss, \DY) -- (\ss, \UY);
\draw (\ss, 0) node[below right]{$\sigma_R$};

\def\s{-0.05}
\draw[thick] 
(0.65, 0.05 + \s) to [out = 270, in = 180]
(0.72, -0.05 + \s) to [out = 0, in = 200]
(0.8, 0.04 + \s) to [out = 20, in = 270]
(0.9, 0.1 + \s) to [out = 90, in = 0]
(0.75, 0.2 + \s) to [out = 180, in = 90]
(0.65, 0.05 + \s);

\fill[black, opacity = .05]
(0.65, 0.05 + \s) to [out = 270, in = 180]
(0.72, -0.05 + \s) to [out = 0, in = 200]
(0.8, 0.04 + \s) to [out = 20, in = 270]
(0.9, 0.1 + \s) to [out = 90, in = 0]
(0.75, 0.2 + \s) to [out = 180, in = 90]
(0.65, 0.05 + \s);

\node at (0.75, 0.1) {$C$};

\draw[thick] (\sf, -\R) -- (\ss, -\R) -- (\ss, \R) -- (\sf, \R) -- cycle; 
\draw (\ss, \R) node[below left]{$\mathcal{R}$};

\end{tikzpicture}
\caption{The rectangle $\mathcal{R}$}
\label{fig:rec}
\end{figure}
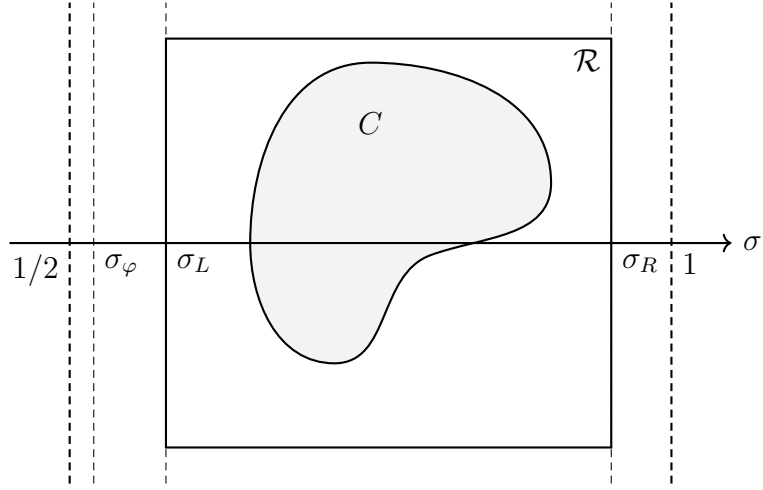

By the Mellin inversion formula for $\lambda(x)$, stated in Lemma \ref{lem:MeTr} (iii), 
we have
\[
\phi_{Y}(z)
= \frac{1}{2 \pi i} \sum_{n = 1}^\infty \int_{c - i \infty}^{c + i \infty} \frac{a_\phi (n) \widehat{\lambda} (\xi) }{n^{z}} \left( \frac{n}{X} \right)^{ - \xi } d \xi
= \frac{1}{2 \pi i} \int_{c - i \infty}^{c + i \infty} \phi(z + \xi) \widehat{\lambda} (\xi) Y^{\xi} d \xi
\]
for $\Re(z) \geq \sigma_L$ and $c>1$, 
where the interchange of the above summation and the integration is justified by Fubini's theorem and Lemma \ref{lem:MeTr} (ii).
We shift the path of integration to the line $\Re(\xi) = - \Re(z) + \sigma_L - \delta $ with $\delta =(\sigma_L - \sigma_\phi)/2$ (see Figure \ref{fig:path2} below). 

\begin{figure}[H]
\centering
\def\xs{5}
\begin{tikzpicture}[scale = 1.2, xscale = \xs]
\def\RX{2.1}
\def\LX{0.4}
\def\UY{2.8}
\def\DY{-1}
\def\sphi{0.52}
\def\sf{0.58} 
\def\ss{0.95} 
\draw[->, opacity = .5] (\LX, 0) -- (\RX, 0) node[right, opacity = .5]{$\Re(z)$};
\draw[densely dashed, opacity = .5] (\sphi, \DY) -- (\sphi, \UY);
\draw (0.52, 0) node[below left, opacity = .5]{$\sigma_\phi$};
\draw[densely dashed, opacity = .5] (\sf, \DY) -- (\sf, \UY);
\draw (\sf, 0) node[below right, opacity = .5]{$\sigma_L$};
\draw[densely dashed, opacity = .5] (\ss, \DY) -- (\ss, \UY);
\draw (\ss, 0) node[below left, opacity = .5]{$\sigma_R$};
\draw[densely dashed, opacity = .5] (1, \DY) -- (1, \UY);
\draw (1, 0) node[below right, opacity = .5]{$1$}; 
\def\zx{0.83}
\def\zy{0.45}
\begin{scope}[shift = {(\zx, \zy)}, xscale = 1/\xs]
\fill[opacity = .5] (0, 0) circle(1.5pt) node[above right, opacity = .5]{$z$};
\draw (0, 0) node[below left]{$O$}; 
\end{scope}
\draw[thick, ->] (\LX, \zy) -- (\RX, \zy) node[right]{$\Re(\xi)$}; 
\draw[thick, ->] (\zx, \DY) -- (\zx, \UY) node[above]{$\Im(\xi)$}; 
\draw[very thick] (\zx + 1.1, \DY) -- (\zx + 1.1, \UY); 
\draw[very thick] (\sphi/2 + \sf/2, \DY) -- (\sphi/2 + \sf/2, \UY); 
\draw[very thick, ->] (\zx + 1.1, \zy) -- (\zx + 1.1, \zy + 0.6) node[right]{$(c)$}; 
\draw[very thick, ->] (\sphi/2 + \sf/2, \zy) -- (\sphi/2 + \sf/2, \zy + 0.6) ;
\draw (\sphi/2 + \sf/2 - 0.02, \zy + 0.6) node[left]{$(- \Re(z) + \sigma_L- \delta)$}; 
\begin{scope}[shift = {(1.6, 1.6)}, xscale = 1/\xs]
\fill (0, 0) circle(1.5pt) node[above]{$\xi$};
\draw (0, 0) node[below, opacity = .5]{$z + \xi$};
\end{scope}
\begin{scope}[shift = {(1, 0)}, xscale = 1/\xs]
\fill (0, 0) circle(1.5pt);
\end{scope}
\draw[-{Stealth[length=2mm]}]
(1.35, -0.5) node[below]{$\xi = 1 - z$} to [out = 90, in = 0]
(1.1, 0.15) to [out = 190, in = 70]
(1.01, 0.02);
\end{tikzpicture}
\caption{The shift of the path of integration}
\label{fig:path2}
\end{figure}
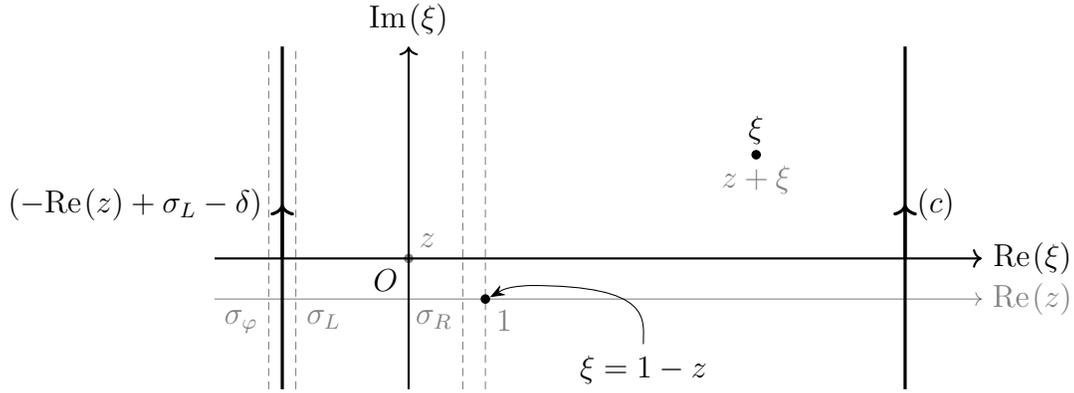

By Lemma \ref{lem:MeTr} (ii), and the axiom $(S3)$ of $\widetilde{\mathcal{S}}$, 
we can shift the above path of integration to deduce
\begin{equation}\label{eqn:P1}
\phi(z) - \phi_{X}(z) 
= - \frac{1}{2 \pi i} \int_{- \Re(z) + \sigma_L - \delta - i \infty }^{ - \Re(z) + \sigma_L - \delta + i \infty} \phi(z + \xi) \widehat{\lambda}(\xi) X^{\xi} d \xi - E(z),
\end{equation}
where $E(z)$ is given by \eqref{eqn:RE} and the estimate $\displaystyle E(z) \ll X \left( | \Im (z) | + 2 \right)^{-2}$ uniformly for $\sigma _L \leq \Re(z) \leq \sigma_R$ by \eqref{eqn:E-res}. 

By Cauchy's integral formula, we have
\[
\phi(s + i \tau) - \phi_X(s + i \tau)
 = \frac{1}{2 \pi i} \int_{\partial \mathcal{R}} \frac{ \phi(z + i \tau) - \phi_X(z + i \tau) }{z - s} dz
\]
for $s \in C$ and $\tau \in \mathbb{R}$, 
where $\partial A$ denotes the boundary of a set $A$, and the path of integration $\partial \mathcal{R}$ is oriented counterclockwise. 
Hence we have
\begin{align*}
&\frac{1}{T} \int_{0}^{T} \max_{s \in C} \left| \phi(s + i \tau) - \phi_X (s + i \tau) \right| d \tau \\
\ll&\, \frac{1}{\dist (C, \partial \mathcal{R})} \cdot \frac{1}{T} \int_{\partial \mathcal{R}} \left| dz \right| \int_{0}^{T} \left|  \phi(z + i \tau) - \phi_X (z + i \tau) \right| d \tau \\
\ll& \frac{\ell(\partial \mathcal{R})}{\dist (C, \partial \mathcal{R})} \sup_{\sigma_L \leq \sigma \leq \sigma_R} \left( \frac{1}{T}\int_{-2T}^{2T} \left| \phi(\sigma + i t) - \phi_X (\sigma + i t) \right| dt \right),
\end{align*}
where $\dist(C, \partial \mathcal{R})$ denotes the distance between $C$ and $\partial \mathcal{R}$, 
and $\ell( \partial \mathcal{R} )$ is the length of $\partial \mathcal{R}$. 
Now, by the equation \eqref{eqn:P1}, 
we have
\begin{align*}
&\phi(\sigma + i t) - \phi_X (\sigma + i t) \\
\ll&\, X^{- \delta} \int_{- \infty}^\infty \left| \phi \left( \sigma_L - \delta + i (t + v) \right) \right| \left| \widehat{\lambda}( - \sigma + \sigma_L - \delta + i v ) \right| d v
+ \left| E(\sigma + i \tau) \right|
\end{align*}
uniformly for $\sigma_L \leq \sigma \leq \sigma_R$. 
Thus we find
\begin{align*}
&\frac{1}{T} \int_{-2T}^{2T} \left| \phi(\sigma + i t) - \phi_X (\sigma + i t) \right| dt \\ 
\ll &\, X^{- \delta} \int_{- \infty}^{\infty} \left| \widehat{\lambda}( - \sigma + \sigma_L - \delta + i v ) \right| dv \left\{ \frac{1}{T} \int_{-2T}^{2T} \left| \phi \left( \sigma_L - \delta + i (t + v) \right) \right| dt \right\} \\
& \quad + \frac{1}{T} \int_{-2T}^{2T} \left| E(\sigma + i \tau) \right| dt \\
=:&\, S_1 + S_2
\end{align*}
uniformly for $\sigma_L \leq \sigma \leq \sigma_R$. 

We estimate $S_1$.
Since the estimate
\begin{align*}
\frac{1}{T} \int_{-2T}^{2T} \left| \phi \left( u + i (t + v) \right) \right| dt 
&\leq \frac{1}{T} \left( \int_{-2T}^{2T} \left| \phi \left( u + i (t + v) \right) \right|^2 dt \right)^{1/2} \left( \int_{0}^{2T} 1^2 dt \right)^{1/2} \\
&\ll \left( \frac{1}{T}  \int_{-2T - |v|}^{2T + |v|} \left| \phi \left( u + i t \right) \right|^2 dt \right)^{1/2} \\
&\ll \left( \frac{1}{T} \left( 4T + 2|v| \right) \right)^{1/2} \\
&\ll 1 + |v|
\end{align*}
holds uniformly for $\sigma_L - \delta \leq u \leq \sigma_R - \delta$ by Lemma \ref{lem:UME}, 
we have
\[
S_1
\ll X^{- \delta} \int_{- \infty}^{\infty} \left| \widehat{\lambda}( - \sigma + \sigma_L - \delta + i v ) \right| \left( 1 + |v| \right) dv
\ll X^{- \delta}
\]
uniformly for $\sigma_L \leq \sigma \leq \sigma_R$ by using Lemma \ref{lem:MeTr} (ii).

We estimate $S_2$. 
By the estimate \eqref{eqn:E-res}, 
we have $S_2 \ll X T^{-1}$ uniformly for $\sigma_L \leq \sigma \leq \sigma_R$. 

From the estimates $S_1$ and $S_2$, we deduce the conclusion. 
\end{proof}

Next, we consider the mean value theorem for the random model of the element in $\widetilde{\mathcal{S}}$.
For any $\omega \in \Omega$ and $n \in \mathbb{N}$, let $\omega(1) = 1$ and let
\[
\omega(n) = \prod_{j = 1}^k \omega(p_j)^{r_j}
\]
for $n \geq 2$, 
where $n = p_1^{r_1} \cdots p_k^{r_k}$ is the prime factorization of $n$. 
Then we have the following lemma:

\begin{lemma}\label{lem:RM}
Let $\phi \in \widetilde{\mathcal{S}}$. 
We have the following:
\begin{enumerate}[label=\rm{(\roman*)}]

\item The equation $\phi(s, \omega) = \sum_{n  = 1}^{\infty} a_\phi (n) \omega(n) n^{-s}$ holds for $\sigma>1/2$ almost surely.

\item The series $\phi(s, \omega)$ is holomorphic for $\sigma>1/2$ almost surely.

\item Let $\sigma_1 > \sigma_\phi$.
Then, for almost sure $\omega \in \Omega$, 
there exists $C(\sigma_1, \omega) > 0$ such that 
\[
\left| \phi(s, \omega) \right| 
\leq C(\sigma_1, \omega) \left( |t| + 2 \right)
\]
for $\Re(s) \geq \sigma_1$

\item Let $\sigma_1 > \sigma_\phi$. 
The expectation $\mathbb{E}^{\bm{m}} \left[ \left| \phi(s, \omega) \right| \right]$ is bounded uniformly for $\sigma \geq \sigma_1$

\end{enumerate}
\end{lemma}

\begin{proof}
For the statements (i) and (ii), the proof can be seen in \cite[Lemma 4.1 and Lemma 4.2]{S2007}.

Next, we show the statement (iii). 
Let $\sigma_m = (\sigma_\phi + \sigma_1)/2$ and put 
\[
S_u(\omega) 
= \sum_{n \leq u} \frac{a_\phi (n) \omega(n)}{n^{\sigma_m}}.
\]
for $u \geq 1$. 
By the statement (i), $S_u(\omega)$ is bounded almost surely. 
Hence, for almost all $\omega \in \Omega$, 
there exists $C_0(\sigma_1, \omega)>0$ such that $\left| S_u \left( \omega \right) \right| \leq C_0(\sigma_1, \omega)$. 
We fix such $\omega \in \Omega$. 
Then, by partial summation, we have
\[
\phi \left(s, \omega \right)
= \int_{1-}^{\infty} \frac{d S_u \left( \omega \right)}{u^{s - \sigma_m}}
=(s - \sigma_m) \int_2^\infty \frac{ S_u\left( \omega \right) \,d u}{u^{s + 1 - \sigma_m}}
\ll_{\sigma_1, \omega} \left| t \right| + 2
\]
for $\sigma \geq \sigma_1$, which complete the proof of the statement (ii). 

Finally, we show the statement (iii). 
By the Cauchy-Schwarz inequality, 
we have
\[
\mathbb{E}^{\bm{m}} \left[ \left| \phi(s, \omega) \right| \right]
\leq \left( \mathbb{E}^{\bm{m}} \left[ \left| \phi(s, \omega) \right|^2 \right] \right)^{1/2}
=  \left( \sum_{n = 1}^\infty \frac{\left| a_\phi(n) \right|^2}{n^{2 \sigma}} \right)^{1/2}
\leq \left( \sum_{n = 1}^\infty \frac{\left| a_\phi(n) \right|^2}{n^{2 \sigma_1}} \right)^{1/2}
< \infty
\]
for $\sigma \geq \sigma_1$, 
where the convergence of the last series is deduced from the axiom (i) of $\widetilde{\mathcal{S}}$. 
This completes the proof. 
\end{proof}

For $X \geq 2$, we define the function $\phi_X (s,\, \omega)$ by 
\[
\phi_X (s,\, \omega)
= \sum_{n = 1}^\infty \frac{a_\phi (n) \omega(n) \lambda( n /X )}{n^s}. 
\]

The following proposition is the random model version of Proposition \ref{prop:MT}. 

\begin{proposition}\label{prop:LMRM}
Let $\phi(s)$ be an element of $\widetilde{\mathcal{S}}$ and $C$ be a compact subset of $D_{\phi}$. 
Then we have
\[
\lim_{X \rightarrow \infty} \mathbb{E}^{\bm{m}} \left[ \max_{s \in C} \left| \phi(s,\, \omega) - \phi_X (s,\, \omega) \right| \right] = 0. 
\]
\end{proposition}

\begin{proof}
We fix a rectangle $\mathcal{R} = [\sigma_L, \sigma_R] \times i [ - R, R ]$ satisfying
\[
\sigma_\phi < \sigma_L < \min_{s \in C} \Re(s) \leq \max_{s \in C} \Re(s) < \sigma_R < 1
\AND R > |\max_{s \in C} \Im(s)| \lor |\min_{s \in C} \Im(s)|,
\]
where the figure of $\mathcal{R}$ can be seen in Figure \ref{fig:rec}.
Suppose that $\omega \in \Omega$ satisfies the statement (ii) and (iii) in Lemma \ref{lem:RM}. 
By the Mellin inversion formula for $\lambda(x)$, as is stated in Lemma \ref{lem:MeTr} (iii), 
we have
\[
\phi_X(z, \omega)
= \frac{1}{2 \pi i} \sum_{n = 1}^\infty \int_{c - i \infty}^{c + i \infty} \frac{a_\phi (n) \omega(n) \widehat{\lambda} (\xi) }{n^{z}} \left( \frac{n}{X} \right)^{ - \xi } d \xi
= \frac{1}{2 \pi i} \int_{c - i \infty}^{c + i \infty} \phi(z + \xi, \omega) \widehat{\lambda} (\xi) X^{\xi} d \xi
\]
for $z \in D_\phi$ and $c > 1$, 
where the interchange of the above summation and the integration is justified by Fubini's theorem and Lemma \ref{lem:MeTr} (ii).
We shift the path of integration to the line $\Re(\xi) = - \Re(z) + \sigma_L - \delta $ with $\delta =(\sigma_L - \sigma_\phi)/2$, where the figure of the path of integration can be seen in Figure \ref{fig:path2}.
By Lemma \ref{lem:MeTr} (ii), and Lemma \ref{lem:RM} (iii), 
we can shift the above path of integration to deduce
\begin{equation}\label{eqn:P2}
\phi(z, \omega) - \phi_{X}(z, \omega) 
= - \frac{1}{2 \pi i} \int_{- \Re(z) + \sigma_L - \delta - i \infty }^{ - \Re(z) + \sigma_L - \delta + i \infty} \phi(z + \xi, \omega) \widehat{\lambda}(\xi) X^{\xi} d \xi.
\end{equation}
By Cauchy's integral formula, we have
\[
\phi(s, \omega) - \phi_X(s, \omega)
 = \frac{1}{2 \pi i} \int_{\partial \mathcal{R}} \frac{ \phi(z, \omega) - \phi_X(z, \omega) }{z - s} dz
\]
for $s \in C$. 
Hence we have, by the equation \eqref{eqn:P2}, 
\begin{align*}
& \max_{s \in C} \left| \phi(s,\, \omega) - \phi_X (s,\, \omega) \right| \\
\leq &\, \frac{1}{2 \pi \dist(C,\, \partial \mathcal{R})} \int_{\partial \mathcal{R}} \left| \phi(z, \omega) - \phi_X(z, \omega) \right| \left| d z \right| \\
\leq &\, \frac{X^{- \delta}}{4 \pi^2 \dist(C,\, \partial \mathcal{R})} \int_{\partial \mathcal{R}} \left| d z \right| \int_{- \infty}^{\infty} \left| \phi \left( \sigma_L - \delta + i (\Im(z) + v), \omega \right) \right| \left| \widehat{\lambda}( - \Re(z) + \sigma_L - \delta + i v ) \right| d v.
\end{align*}
Taking the expectation, we obtain
\begin{align*}
& \mathbb{E}^{\bm{m}} \left[ \max_{s \in C} \left| \phi(s,\, \omega) - \phi_X (s,\, \omega) \right| \right] \\
\leq &\, \frac{ \ell\left( \partial \mathcal{R} \right)X^{- \delta}}{4 \pi^2 \dist(C,\, \partial \mathcal{R})} \times \\
&\times \sup_{z \in \partial \mathcal{R}} \int_{- \infty}^{\infty} \mathbb{E}^{\bm{m}} \left[ \left| \phi \left( \sigma_L - \delta + i (\Im(z) + v), \omega \right) \right| \right] \left| \widehat{\lambda}( - \Re(z) + \sigma_L - \delta + i v ) \right| d v.
\end{align*}
By Lemma \ref{lem:RM} (iii) and Lemma \ref{lem:MeTr}, 
we find that the value
\[
\int_{- \infty}^{\infty} \mathbb{E}^{\bm{m}} \left[ \left| \phi \left( \sigma_L - \delta + i (\Im(z) + v), \omega \right) \right| \right] \left| \widehat{\lambda}( - \Re(z) + \sigma_L - \delta + i v ) \right| d v
\]
is bounded uniformly for $z \in \partial \mathcal{R}$. 
Therefore we obtain 
\[
\mathbb{E}^{\bm{m}} \left[ \max_{s \in C} \left| \phi(s,\, \omega) - \phi_X (s,\, \omega) \right| \right]
\ll X^{- \delta}
\rightarrow 0
\] 
as $X \rightarrow \infty$. 
This completes the proof.
\end{proof}

Let $\mathcal{P}_{F}$ be a finite subset of the set of all prime numbers. 
Define the probability measure $\mathbb{H}_T^{\mathcal{P}_F}$ on $\left( \mathbb{T}^{\mathcal{P}_{F}}, \mathcal{B}( \mathbb{T}^{\mathcal{P}_{F}} ) \right)$ by
\[
\mathbb{H}_T^{\mathcal{P}_F} \left( A \right)
= \frac{1}{T} \meas\left\{ \tau \in [0, T] \mid (p^{i \tau})_{p \in \mathcal{P}_F } \in A \right\}
\]
for $A \in \mathbb{T}^{\mathcal{P}_{F}}$. 

On the other hand, since the set $\mathbb{T}^{\mathcal{P}_{F}}$ is a compact Hausdorff space, 
there exists the probability Haar measure $\bm{m}_{\mathcal{P}_F}$ on $\left( \mathbb{T}^{\mathcal{P}_{F}}, \mathcal{B}( \mathbb{T}^{\mathcal{P}_{F}} ) \right)$. 

Then we have the following lemma:

\begin{lemma}\label{lem:HM}
The probability $\mathbb{H}_T^{\mathcal{P}_F}$ converges weakly to $\bm{m}_{\mathcal{P}_F}$ as $T \rightarrow \infty$. 
\end{lemma}

\begin{proof}
For example, the proof can be seen in \cite[in the proof of Lemma 4.4 in p.~68]{S2007}.
\end{proof}

\begin{proof}[Proof of (i) of Theorem \ref{thm:main}.]
Let $X$ be a large number with $X>\max \mathcal{P}_N$. 
Let $\mathbb{P}_T$ denote the probability measure on $\left( [0, T], \mathcal{B}([0, T]) \right)$ given by
\[
\mathbb{P}_T (E) = \frac{1}{T} \meas\left( E \right), \quad E \in \mathcal{B}([0, T]). 
\]
Let $F:\prod_{j = 1}^{r} H(D_{\phi_j}) \times \mathbb{T}^{\mathcal{P}_N} \rightarrow \mathbb{R}$ be a bounded Lipschitz continuous function. 
Then there exist $M(F)$, $C(F)$ such that
\[
\left| F\left(\bm{f}, \bm{x} \right) \right| \leq M(F) \AND \left| F\left( \bm{f}, \bm{x} \right) - F\left( \bm{g}, \bm{y} \right) \right| \leq C(F) d\left( (\bm{f}, \bm{x}) ,\, (\bm{g}, \bm{y}) \right )
\]
for any $(\bm{f}, \bm{x}) = \left(  (f_j)_{j = 1}^r, ( x_{p_{k_n}} )_{n = 1}^N \right),\, (\bm{g}, \bm{y}) = \left(  (g_j)_{j = 1}^r, ( y_{p_{k_n}} )_{n = 1}^N \right) \in \prod_{j = 1}^{r} H(D_{\phi_j}) \times \mathbb{T}^{\mathcal{P}_N}$. 
It is enough to show that $\left| \mathbb{E}^{\nu_{T, \underline{\phi},  \mathcal{P}_N }} [ F ] - \mathbb{E}^{\nu_{\underline{\phi},\, \mathcal{P}_N }} [ F ] \right| \rightarrow 0$ as $T \rightarrow \infty$ by a property of weak convergence (see e.g. \cite[Theorem 3.9.1]{D2010}). 
For brevity, we write
\[
\bm{\gamma}_{\mathcal{P}_N}(\tau) = \left( p^{i \tau} \right)_{p \in \mathcal{P}_N}
\AND
\omega_{\mathcal{P}_N} = \left( \omega(p) \right)_{p \in \mathcal{P}_N}
\]
We find that
\begin{align*}
\left| \mathbb{E}^{\nu_{T, \underline{\phi},  \mathcal{P}_N }} \left[ F \right] - \mathbb{E}^{\nu_{\underline{\phi},\, \mathcal{P}_N }} \left[ F \right] \right|
&= \left| \mathbb{E}^{\mathbb{P}_T} \left[ F \left( \underline{\phi}(s + i \tau),\, \bm{\gamma}_{\mathcal{P}_N}(\tau)   \right) \right] - \mathbb{E}^{\bm{m}} \left[ F \left( \underline{\phi}(s,\, \omega),\, \omega_{\mathcal{P}_N} \right) \right] \right| \\
&\leq \left| \mathbb{E}^{\mathbb{P}_T} \left[ F \left( \underline{\phi}(s + i \tau),\, \bm{\gamma}_{\mathcal{P}_N}(\tau)  \right) \right] - \mathbb{E}^{\mathbb{P}_T} \left[ F \left( \underline{\phi}_{X}(s + i \tau),\, \bm{\gamma}_{\mathcal{P}_N}(\tau) \right) \right] \right| \\
& \quad + \left| \mathbb{E}^{\mathbb{P}_T} \left[ F \left( \underline{\phi}_{X}(s + i \tau),\, \bm{\gamma}_{\mathcal{P}_N}(\tau) \right) \right] - \mathbb{E}^{\bm{m}} \left[ F \left( \underline{\phi}_X(s,\, \omega),\, \omega_{\mathcal{P}_N} \right) \right] \right| \\
& \quad + \left| \mathbb{E}^{\bm{m}} \left[ F \left( \underline{\phi}_X(s,\, \omega),\, \omega_{\mathcal{P}_N} \right) \right] - \mathbb{E}^{\bm{m}} \left[ F \left( \underline{\phi}(s,\, \omega),\, \omega_{\mathcal{P}_N} \right) \right] \right| \\
& =: \Sigma_1(T,\, X) + \Sigma_2(T,\, X) + \Sigma_3(T,\, X)
\end{align*}

First, we estimate the term $\Sigma_1(T,\, X)$.
We have
\begin{align*}
\Sigma_1(T,\, X) 
&\leq C(F)  \mathbb{E}^{\mathbb{P}_T} \left[ d \left( ( \underline{\phi}(s + i \tau),\, \bm{\gamma}_{\mathcal{P}_N}(\tau) ) ,\, ( \underline{\phi}_{X}(s + i \tau),\, \bm{\gamma}_{\mathcal{P}_N}(\tau) ) \right) \right] \\
&\leq C(F) \sum_{j =1}^r \sum_{\ell = 1}^\infty \frac{ \mathbb{E}^{\mathbb{P}_T} \left[ d_{\ell,\, \phi_j} \left( \phi_j (s + i \tau),\, {\phi_j}_{, X} (s + i \tau) \right) \right] \land 1 }{2^{\ell}}.
\end{align*}
As we take the limit $\lim_{T \rightarrow \infty} \lim_{X \rightarrow \infty}$, 
the above tends to 0 by Proposition \ref{prop:MT} and Lebesgue's dominated convergence theorem. 

Next, we estimate $\Sigma_2(T,\, X)$. 
Put
\[
\mathcal{P}( \lambda,\, X ) 
=\left\{ p \mid \textrm{$p$ divides $\displaystyle \prod_{\substack{n \in \mathbb{N}; \\ \lambda(n/X) \neq 0 }} n $} \right\}.
\]
Note that the inclusion $\mathcal{P}_N \subset \mathcal{P}( \lambda,\, X )$ holds if $X$ is large enough.
We define the continuous function $\Phi_X : \mathbb{T}^{\mathcal{P}( \lambda,\, X )}  \rightarrow \prod_{j = 1}^r H(D_{\phi_j}) \times \mathbb{T}^{\mathcal{P}_N} $ by
\[
\Phi_X(\omega) = \left( \left( \phi_{j,\, X} (\omega)  \right)_{j = 1}^r,\, (\omega(p))_{p \in \mathcal{P}_N} \right)
\]
for $\omega = \left( \omega(p) \right)_{p \in \mathcal{P}( \lambda,\, X ) } \in \mathbb{T}^{\mathcal{P}( \lambda,\, X )}$.
Here, $\phi_{j ,\, X}(\omega)$ is given by
\[
\phi_{j ,\, X}(\omega)
= \sum_{\substack{n = 1; \\ \lambda(n/X) \neq 0 }}^{\infty} \frac{ a_{\phi_j}(n) \lambda(n/X) }{n^s} \prod_{p | n} \omega(p)^{ - \nu(p;n)} 
\]
for $j = 1, \ldots, r$, and the symbol $\nu(p;n)$ stands for the exponent of $p$ in the prime factorization of $n$.
Then we have
\begin{align*}
&\mathbb{E}^{\mathbb{P}_T} \left[ F \left( \underline{\phi}_{X}(s + i \tau),\, \bm{\gamma}_{\mathcal{P}_N}(\tau) \right) \right]
= \mathbb{E}^{ \mathbb{H}_{T}^{ \mathcal{P}( \lambda,\, X ) } \circ {\Phi_X}^{-1}} \left[ F \right] \\
\rightarrow&\, \mathbb{E}^{ \bm{m}_{\mathcal{P}( \lambda,\, X )} \circ {\Phi_X}^{-1} } \left[ F \right] 
=\mathbb{E}^{\bm{m}} \left[ F \left( \left( \underline{\phi}_X(s,\, \omega),\, \omega_{\mathcal{P}_N} \right) \right) \right]
\end{align*}
as $T \rightarrow \infty$ by Lemma \ref{lem:HM} and a property of weak convergence (see e.g. \cite[Section 2, The Mapping Theorem]{B1999}). 
Hence we obtain $\lim_{T \rightarrow \infty} \Sigma_2(T, X) = 0$. 

Finally, we estimate $\Sigma_3(T, X)$. 
We have
\begin{align*}
\Sigma_3(T, X) 
\leq&\, C(F) \mathbb{E}^{\bm{m}} \left[ d \left( \left( \underline{\phi}_X(s,\, \omega),\, \omega_{\mathcal{P}_N} \right),\, \left( \underline{\phi} (s,\, \omega),\, \omega_{\mathcal{P}_N} \right) \right) \right] \\
\leq &\,  C(F) \sum_{j =1}^r \sum_{\ell = 1}^\infty \frac{ \mathbb{E}^{\bm{m}} \left[ d_{\ell,\, \phi_j} \left( \phi_{j, X} (s,\, \omega),\, {\phi_j} (s,\, \omega) \right) \right] \land 1 }{2^{\ell}},
\end{align*}
and the right-hand side of the above inequality tends to $0$ as $X \rightarrow \infty$ by Proposition \ref{prop:LMRM} and Lebesgue's dominated convergence theorem. 
This completes the proof.
\end{proof}

\section{Proof of (ii) of Theorem \ref{thm:main}}

In the following, we define
\[
\log \phi(s, \omega) = \sum_{p} \log \phi_p (s, \omega(p))
\]
for $\phi \in \widetilde{\mathcal{S}}$ and $\omega \in \Omega$, 
where we put 
\[
\log \phi_p (s, z) = - \sum_{j = 1}^{m_\phi} \Log \left( 1 - \frac{\alpha_{j, \phi}(p) z}{ p^{s}} \right)
\]
for $z \in \mathbb{T}$.
Here, $\Log(z)$ stands for the principal branch of logarithm.
We can check that
\[
\mathbb{E}^{\bm{m}} \left[ \log \phi_p (s, \omega(p)) \right] =0 \AND 
\sum_p \mathbb{E}^{\bm{m}} \left[ \left| \log \phi_p (s, \omega(p)) \right|^2 \right] < \infty
\]
by using the axioms of $\widetilde{\mathcal{S}}$, 
where we use the estimate $\left| \alpha_{g, \phi_j}(p) \right| \leq 1$. 
This estimate is proved in Lemma 2.2 in \cite{S2007}. 
Hence we deduce from Kolmogorov's theorem (see e.g. \cite[Theorem B.10.1]{K2021}) that $\log \phi(s, \omega)$ converges almost everywhere. 
Noting that the equation
\[
L(s, \omega) = \exp \left( \log \phi(s, \omega) \right)
\]
holds almost everywhere, 
we first consider the support of $\log \phi(s, \omega)$. 

\begin{lemma}\label{lem:dense}
Let the assumption be the same as in Theorem \ref{thm:main}.
Then, the set of all convergent element
\[
\left( \left( \log \phi_j (s, \omega) \right)_{j = 1}^r,\, \left( \omega(p) \right)_{p \in \mathcal{P}_N} \right)
\]
is dense in $\prod_{j = 1}^{r} H(D_{\phi_j}) \times \mathbb{T}^{\mathcal{P}_N}$. 

\end{lemma}

\begin{proof}
Let $\left( (f_j (s) )_{j = 1}^r,\, \left( z(p) \right)_{p \in \mathcal{P}_N} \right) \in \prod_{j = 1}^{r} H(D_{\phi_j}) \times \mathbb{T}^{\mathcal{P}_N}$ and $\epsilon > 0$. 
Let $X$ be a positive parameter. 
For $w_X = \left( w(p) \right)_{p > X} \in \prod_{p > X} \mathbb{T}_p$, 
we have
\[
h_{j, X}(s, w_X)
:=\sum_{p > X} \left( \log \phi_{j, p} (s,\, w(p)) - \frac{ a_{\phi_{j}}(p) w(p) }{p^s} \right)
= \sum_{p > X} \sum_{g = 1}^{m_{\phi_j}} \sum_{k = 2}^\infty \frac{ \alpha_{g, \phi_{j}}(p)^k w(p)^k }{ k p^{ks} }
\]
for $j = 1, \ldots, r$ by Lemma 2.2 in \cite{S2007}.
Using the estimate $\left| \alpha_{g, \phi_j}(p) \right| \leq 1$, 
we have
\[
\left \| h_{j, X} \right \|_\infty
= \sup_{ w_X \in \prod_{p > X} \mathbb{T}_p} \sup_{ s \in D{\phi_j} } \left| h_{j, X}(s, w_X) \right|
\leq m_{\phi_j} \sum_{p > X} \sum_{k = 2}^\infty \frac{1}{k p^{k \sigma_{\phi_j}}}
\rightarrow 0
\]
as $X \rightarrow \infty$. 
Hence we can take $X_0 > \max \mathcal{P}_N$ such that $\left \| h_{j, X_0} \right \|_\infty < \epsilon / (2r)$ for any $j = 1, \ldots, r$.

For any $p \leq X_0$, 
we put
\[
c_0 (p)
= 
\begin{dcases}
z(p) & \textrm{if $p \in \mathcal{P}_N$}, \\
1 & \textrm{otherwise}. 
\end{dcases}
\]
By the assumption, 
there exists a convergent element 
\[
\left( \sum_{p > X} \frac{a_{\phi_1}(p) c_0(p)}{p^{s}}, \ldots, \sum_{p > X} \frac{a_{\phi_r}(p) c_0(p)}{p^{s}} \right),\, c_0(p) \in \mathbb{T}
\]
such that
\[
d_{\phi_j} \left( f_j(s) - \sum_{p \leq X_0} \log \phi_{j, p} \left( s, c_0(p) \right),\, \sum_{p > X} \frac{a_{\phi_{j, r}}(p) c_0(p)}{p^{s}} \right)
< \frac{\epsilon}{2 r}
\]
for $j = 1, \ldots, r$. 
Putting $c_0 = \left( c_0 (p) \right)_p \in \Omega$, 
we have
\begin{align*}
& d \left( \left( \left( \log \phi_j (s, c_0) \right)_{j = 1}^r,\, \left( c_0(p) \right)_{p \in \mathcal{P}_N} \right),\, \left( (f_j (s) )_{j = 1}^r,\, \left( z(p) \right)_{p \in \mathcal{P}_N} \right) \right) \\
=&\, \sum_{j = 1}^r d_{\phi_j} \left( \log \phi_j (s, c_0),\, f_j(s) \right) \\
\leq&\, \sum_{j = 1}^r d_{\phi_j} \left( \sum_{p \leq X_0} \log \phi_{j, p} (s, c_0 (p)) + \sum_{p > X} \frac{a_{\phi_{j, r}}(p) c_0(p)}{p^{s}} ,\, f_j(s) \right)
+ \sum_{j = 1}^r \left \| h_{j, X_0} \right \|_\infty \\
<&\, \epsilon. 
\end{align*}
This completes the proof. 
\end{proof}

\begin{proposition}\label{prop:supp}
Let $\mathcal{S}$ be a separable metric space and let $\mathcal{B}_j$ be a separable Fr\'echet space for $j = 1, \ldots, r$.
For any $j = 1, \ldots, r$, given a sequence $(F_{n,j})_{n = 1}^\infty$ of $\mathcal{B}_j$-valued continuous functions on $\mathcal{S}$. 
Suppose that $\mathcal{M}$ is a finite subset of the set of positive integers and that $X = (X_n)_{n=1}^\infty$ is a sequence of independent $\mathcal{S}$-valued random variables on a probability space $(\Omega, \mathcal{A}, \mathbb{P})$ such that the series $\sum_{n = 1}^\infty F_{n, j} (X_n)$ converges almost surely for $j = 1, \ldots, r$. 
We define the $ \prod_{j = 1}^r \mathcal{B}_j \times \prod_{n \in \mathcal{M}} \mathcal{S}_n $-valued random variable $\Phi$ by
$\Phi (X) = \left( \left( \sum_{n = 1}^\infty F_{n,j}(X_n) \right)_{j = 1}^r,\, (X_n)_{n \in \mathcal{M}} \right)$, where $\mathcal{S}_n = \mathcal{S}$. 
Then the support of the distribution of $\Phi$ coincide with the closure of the set
\[
\left\{ \left( \left( \sum_{n = 1}^\infty F_{n, j}(x_n) \right)_{j = 1}^r ,\, (x_n)_{n \in \mathcal{M}} \right) ~;~ x_n \in \supp( X_n ) \quad {\rm and} \quad \textit{$\displaystyle \sum_{n = 1}^\infty F_{n, j}(x_n)$ {\rm converges}} \right\},
\]
where $\supp(Y)$ stands for the support of the distribution of a random variable $Y$. 
\end{proposition}

\begin{proof}
Let $d_\mathcal{S}$ denote the metric on $\mathcal{S}$ and let $d_{\mathcal{B}_j}$ be the metric on $\mathcal{B}_j$ induced from a countable family of semi-norms for $j = 1, \ldots, r$. 
Note that the distance $d_\mathcal{B}$ has the translation-invariant property, that is, 
$d_{\mathcal{B}_j}(f, g) = d_{\mathcal{B}_j}(f + h, g + h)$ for $f, g, h \in \mathcal{B}_j$. 
Put $\mathcal{B} = \prod_{j =1}^r \mathcal{B}_j$ and define a metric $d_{\mathcal{B}}$ on $\mathcal{B}$ by
\[
d_{\mathcal{B}} (f, g) = \sum_{j = 1}^r d_{\mathcal{B}_j}( f_j, g_j ), \quad 
\underline{f} = \left( f_j \right )_{j = 1}^r,\, \underline{g} = \left( g_j \right )_{j = 1}^r \in \mathcal{B}.
\]
For brevity, 
we write 
\[F_j(\xi) = \sum_{n = 1}^\infty F_{n, j}( \xi_n ) \AND \underline{F} (\xi) = \left( F_j(\xi) \right)_{j = 1}^r
\]
for $\xi = (\xi)_{n = 1}^\infty$.
We also put
\[
\mathscr{X}
= \left\{ \left( \underline{F}(x) ,\, \left( x_n \right)_{n \in \mathcal{M}} \right) ~;~ x = \left( x_n \right)_{n = 1}^\infty \in \prod_{n = 1}^\infty \supp( X_n ) \quad {\rm and} \quad \textit{$\displaystyle \underline{F}(x)$ {\rm converges}} \right\}.
\]

We first prove the inclusion $\mathscr{X} \subset \supp \left( \Phi \right)$. 
We fix an element $\left( \underline{F}(x),\, (x_n)_{n \in \mathcal{M}} \right) \in \mathscr{X}$ and $\epsilon > 0$ arbitrarily, 
where $x = \left( x_n \right)_{n = 1}^\infty \in \prod_{n = 1}^\infty \supp( X_n )$.
Let $U_{d_{ \mathcal{B} }} (\underline{f}; \delta)$ denote the open ball of radius $\delta$ centered at $\underline{f}$ in the space $\mathcal{B}$, 
and let $U_{d_\mathcal{S}}(a; \delta)$ denote the open ball of radius $\delta$ centered at $a$ in the space $\mathcal{S}$.
Then we have
\begin{align*}
&\mathbb{P} \left( \underline{F} (X) \in \prod_{j = 1}^r U_{d_{ \mathcal{B} }} \left( \underline{F} (x) ~;~ \epsilon \right) \AND \left( X_n \right)_{n \in \mathcal{M}} \in \prod_{n \in \mathcal{M}} U_{d_\mathcal{S}} \left( x_n ; \epsilon \right) \right) \\
\geq&\, \mathbb{P} \left( \textrm{the condition (P) holds} \AND \left( X_n \right)_{n \in \mathcal{M}} \in \prod_{n \in \mathcal{M}} U_{d_\mathcal{S}} \left( x_n ; \epsilon \right) \right),
\end{align*}
where the condition (P) requires the following three: 
\begin{itemize}
\item[(P1)] $d_{\mathcal{B}}\left( \underline{F}(X),\, \underline{S}_N \left( (X_n)_{n =1}^N \right) \right) < \epsilon/3$, 
which is equivalent to $\underline{R}_N \left( (X_n)_{n = N + 1}^\infty \right) \in U_{\mathcal{B}} \left( \bm{0} ; \epsilon/3 \right)$, 
\item[(P2)] $d_{\mathcal{B}} \left( \underline{S}_N \left( (X_n)_{n = 1}^N \right),\, \underline{S}_N\left( (x_n)_{n =1}^N \right) \right) < \epsilon/3$, and
\item[(P3)] $d_{\mathcal{B}} \left( \underline{S}_N \left( (x_n)_{n = 1}^N \right),\, \underline{F} (x) \right) < \epsilon/3$. 
\end{itemize}
Here $\underline{S}_N( \xi_N )$ and $\underline{R}_N( \xi^N )$ is defined by $\underline{S}_N( \xi_N ) = \sum_{n = 1}^N F( \xi_n ) $ and $\underline{R}_N( \xi^N ) = \sum_{n = N + 1}^\infty ( \xi_n ) $ for $\xi_N = (\xi_n)_{n = 1}^N \in \mathbb{R}^N$, $\xi^N = (\xi_n)_{n = N + 1}^\infty \in \mathbb{R}^{\{ N + 1, N + 2, \ldots \}}$ and $N \in \mathbb{N}$. 
We note that, as for the condition $(P1)$, the almost surely convergence of $\underline{F} (X)$ yields 
\[
\lim_{N \rightarrow \infty } \mathbb{P} \left( \underline{R}_N \left( (X_n)_{n = N + 1}^\infty \right) \in U_{\mathcal{B}} \left( \bm{0} ; \epsilon/3 \right) \right) = 1.
\]
The condition $(P3)$ always holds for $N$ large enough. 
Hence there exists $N_0$ with $N_0 \geq \max \mathcal{M} + 1$ such that 
\[
\mathbb{P} \left( \underline{R}_{N_0} \left( (X_n)_{n = N_{0} + 1}^\infty \right) \in U_{\mathcal{B}} \left( \bm{0} ; \epsilon/3 \right) \right) > 0 
\AND 
d_{\mathcal{B}} \left(\underline{S}_{N_0} \left( (x_n)_{n = 1}^{N_0} \right),\, \underline{F} (x) \right) < \epsilon/3.
\]
By the independence property of $X_n$ and the choice of $N_0$, we have the inequality
\begin{align*}
&\mathbb{P} \left( \underline{F} (X) \in U_{d_{\mathcal{B}}} \left( \underline{F} (x) ; \epsilon \right) \AND \left( X_n \right)_{n \in \mathcal{M}} \in \prod_{n \in \mathcal{M}} U_{d_\mathcal{S}} \left( x_n ; \epsilon \right) \right) \\
\geq&\, \mathbb{P} \Bigg( \underline{R}_{N_0} \left( (X_n)_{n = {N_0} + 1}^\infty \right) \in U_{d_\mathcal{B}} \left( \bm{0} ; \epsilon/3 \right) ,\, 
\underline{S}_{N_0} \left( (X_n)_{n = 1}^{N_0} \right) \in U_{d_{\mathcal{B}}} \left( \underline{S}_{N_0} \left( (x_n)_{n =1}^{N_0} \right) ; \epsilon/3 \right) \\
& \quad \quad \AND \left( X_n \right)_{n \in \mathcal{M}} \in \prod_{n \in \mathcal{M}} U_{d_\mathcal{S}} \left( x_n ; \epsilon \right) \Bigg) \\
=&\, \mathbb{P} \left( \underline{R}_{N_0} \left( (X_n)_{n = {N_0} + 1}^\infty \right) \in U_{d_{\mathcal{B}}} \left( \bm{0} ; \epsilon/3 \right) \right) \\
& \times \mathbb{P} \left( \underline{S}_{N_0} \left( (X_n)_{n = 1}^{N_0} \right) \in U_{d_{\mathcal{B}}} \left(S_{N_0} \left( (x_n)_{n =1}^{N_0} \right) ; \epsilon/3 \right) 
\AND \left( X_n \right)_{n \in \mathcal{M}} \in \prod_{n \in \mathcal{M}} U_{d_\mathcal{S}} \left( x_n ; \epsilon \right) \right).
\end{align*}

Now we deduce from the continuity of the function $\underline{S}_{N_0} : \mathcal{S}^{N_0} \rightarrow \mathcal{B}$ that there exists $\delta_{N_0} > 0$ such that $(\xi_n)_{n = 1}^{N_0} \in \prod_{n = 1}^N U_{d_\mathcal{S}}( x_n ; \delta_{N_0} )$ implies $\underline{S}_{N_0} \left( (\xi_n)_{n = 1}^{N_0} \right) \in U_{d_{\mathcal{B}}} \left(S_{N_0} \left( (x_n)_{n =1}^{N_0} \right) ; \epsilon/3 \right) $. 
Therefore, letting $\epsilon_0 = \min\{ \epsilon/3,\, \delta_{N_0} \}$, 
we obtain
\begin{align*}
&\mathbb{P} \left( \underline{F} (X) \in U_{d_{\mathcal{B}}} \left( \underline{F} (x) ; \epsilon \right) \AND \left( X_n \right)_{n \in \mathcal{M}} \in \prod_{n \in \mathcal{M}} U_{d_\mathcal{S}} \left( x_n ; \epsilon \right) \right) \\
\geq&\, \mathbb{P} \left( \underline{R}_{N_0} \left( (X_n)_{n = {N_0} + 1}^\infty \right) \in U\left( \bm{0} ; \epsilon/3 \right) \right)  \times \prod_{n = 1}^{N_0} \mathbb{P} \left( X_n \in U_{d_{\mathcal{S}}} (x_n; \epsilon_0) \right) >0.
\end{align*}
This completes the proof of the inclusion $\mathscr{X} \subset \supp \left( \Phi \right)$. 
Since the set $\supp \left( \Phi \right)$ is closed, we have $\overline{\mathscr{X}} \subset \supp \left( \Phi \right)$, 
where $\overline{}A$ stands for the closure of the set $A$.

Next we prove the inverse inclusion $\overline{\mathscr{X}} \supset \supp \left( \Phi \right)$. 
Let $(\underline{f}, (\xi_{n})_{n \in \mathcal{M}}) \in \supp \left( \Phi \right)$ and fix $\epsilon > 0$.
Then the inequality 
\[
\mathbb{P} \left( \underline{F} (X) \in U_{d_\mathcal{B}} \left(\underline{f} ; \epsilon \right) \AND (X_n)_{n \in \mathcal{M}} \in \prod_{n \in \mathcal{M}} U_{d_\mathcal{S}} \left( \xi_n ; \epsilon \right) \right) > 0
\]
holds. 
Since $\mathbb{P} \left( X_n \not \in \supp(X_n) \right) = 0$ for any $n \in \mathbb{N}$, 
we have
\begin{align*}
&\mathbb{P} \left( \underline{F} (X) \in U_{d_\mathcal{B}} \left(\underline{f} ; \epsilon \right) \AND (X_n)_{n \in \mathcal{M}} \in \prod_{n \in \mathcal{M}} U_{d_\mathcal{S}} \left( \xi_n ; \epsilon \right) \right) \\
=&\, \mathbb{P} \left( \underline{F} (X) \in U_{d_\mathcal{B}} \left(\underline{f} ; \epsilon \right),\, (X_n)_{n \in \mathcal{M}} \in \prod_{n \in \mathcal{M}} U_{d_\mathcal{S}} \left( \xi_n ; \epsilon \right) \AND \textrm{$X_n \in \supp(X_n)$ for $n\ \in \mathbb{N}$ }\right).
\end{align*}
Therefore there exists $x = (x_n)_{n = 1}^\infty$ with $x_n \in \supp(X_n)$ such that
\[
\underline{f} \in U_{d_\mathcal{B}} \left( \underline{F} (x) ; \epsilon \right) 
\AND 
(\xi_n)_{n \in \mathcal{M}} \in \prod_{n \in \mathcal{M}} U_{d_\mathcal{S}} \left( x_n ; \epsilon \right),
\]
which deduce $(\underline{f}, (\xi_{n})_{n \in \mathcal{M}}) \in \overline{\mathscr{X}}$.
This completes the proof. 
\end{proof}

\begin{lemma}\label{lem:supplog}
Let the assumption be the same as in Theorem \ref{thm:main} and put 
\[
\log \underline{\phi} (s, \omega)
= \left( \log \phi_1 (s, \omega), \ldots, \log \phi_r (s, \omega) \right)
\]
for $\omega \in \Omega$. 
Then, the support of the distribution of the $\prod_{j = 1}^{r} H(D_{\phi_j}) \times \mathbb{T}^{\mathcal{P}_N}$-valued random variable 
\[
\left( \log \underline{\phi} (s, \omega), \left( \omega(p) \right)_{p \in \mathcal{P}_N} \right
)
\] 
coincides with $\prod_{j = 1}^{r} H(D_{\phi_j}) \times \mathbb{T}^{\mathcal{P}_N}$.
\end{lemma}

\begin{proof}
By Lemma \ref{lem:dense} and Proposition \ref{prop:supp}, 
it suffices to show that the support of the distribution of $\log \phi_p (s, \omega(p))$ coincides with the set $\left\{ \log \phi_p (s, z) \mid z \in \mathbb{T} \right\}$ for $\phi \in \widetilde{\mathcal{S}}$. 

First, we show that $\log \phi_p (s, z_0)$, $z_0 \in \mathbb{T}$ belongs to the set $\supp( \log \phi_p (s, \omega(p)) )$. 
Let $\epsilon > 0$ be fixed.
We have
\begin{align*}
d_{\phi} \left( \log \phi_p (s, z),\, \log \phi_p (s, z_0) \right)
&= \sum_{\ell = 1}^\infty \frac{ d_{\ell, \phi} \left( \log \phi_p (s, z),\, \log \phi_p (s, z_0) \right) \land 1 }{2^{\ell}} \\
&\leq \sum_{\ell = 1}^\infty \frac{ \sup_{s \in K_{\ell,\, \phi}} \left| \log \phi_p (s, z) - \log \phi_p (s, z_0) \right|}{2^{\ell}} \\
&\leq \sum_{\ell = 1}^\infty 2^{- \ell} \sum_{j = 1}^{m_\phi} \sum_{k = 1}^\infty \frac{ \left| \alpha_{j, \phi} (p) \right|^k | z^k - z_0^k | }{k p^{k \sigma_{\phi}}} \\
&\leq m_{\phi} \sum_{k = 1}^\infty \frac{| z - z_0 |}{p^{k \sigma_\phi}}
= \frac{ m_{\phi} | z - z_0 |}{1 - p^{- \sigma_\phi}},
\end{align*}
where we use the estimate $\left| \alpha_{j, \phi} (p) \right| \leq 1$ and $| z^k - z_0^k | \leq k \left| z - z_0 \right|$. 
Hence we have
\begin{align*}
& \bm{m} \left( d_{\phi} \left( \log \phi_p (s, \omega(p) ),\, \log \phi_p (s, z_0) \right) < \epsilon \right) \\
=&\, \int_{0}^{1} \mathbbm{1} \left \{ d_{\phi} \left( \log \phi_p (s, e^{2 \pi i \theta} ),\, \log \phi_p (s, z_0) \right) < \epsilon \right \} d \, \theta \\
\geq& \, \int_{0}^{1} \mathbbm{1} \left \{ \frac{ m_{\phi} | e^{2 \pi i \theta} - z_0 |}{1 - p^{- \sigma_\phi}} < \epsilon \right \} d \, \theta
> 0.
\end{align*}
Hence $\log \phi_p (s, z_0)$ belongs to $\supp( \log \phi_p (s, \omega(p)) )$. 

Next, we show $\supp( \log \phi_p (s, \omega(p)) )$ is a subset of the set $\left\{ \log \phi_p (s, z) \mid z \in \mathbb{T} \right\}$. 
We take $g \in \supp( \log \phi_p (s, \omega(p)) )$. 
Then, by the definition of $\supp( \log \phi_p (s, \omega(p)) )$, 
we have
\[
\bm{m} \left( d_{\phi} \left( \log \phi_p (s, \omega(p)), g(s) \right) < \frac{1}{n} \right) > 0
\]
for any $n \in \mathbb{N}$.
Hence we can take a sequence $\left( z_n \right)_{n = 1}^\infty \subset \mathbb{T}$ such that $d_{\phi} \left( \log \phi_p (s, z_n), g(s) \right) < 1/n$. 
Since $\mathbb{T}$ is compact, there exists a subsequence $\left( z_{n_k} \right)_{k = 1}^\infty$ of $\left( z_n \right)_{n = 1}^\infty$ such that $z_{n_k} \rightarrow z_0$ for some $z_0 \in \mathbb{T}$ as $k \rightarrow \infty$. 
Therefore, using the continuity of the map $\mathbb{T} \ni z \mapsto \log \phi_p (s, z) \in H(D_{\phi})$, 
we have
\[
d_{\phi} \left( \log \phi_p (s, z_0), g(s) \right)
\leq d_{\phi} \left( \log \phi_p (s, z_0), \log \phi_p (s, z_{n_k}) \right) 
+ d_{\phi} \left( \log \phi_p (s, z_{n_k}), g(s) \right)
\rightarrow 0
\]
as $k \rightarrow \infty$. 
Thus we conclude that $\log \phi_p (s, z_0) = g(s)$. 
This completes the proof. 
\end{proof}

\begin{proof}[Proof of (ii) of Theorem \ref{thm:main}]
We define the map $\Phi : \prod_{j = 1}^{r} H(D_{\phi_j}) \times \mathbb{T}^{\mathcal{P}_N} \rightarrow \prod_{j = 1}^{r} H(D_{\phi_j}) \times \mathbb{T}^{\mathcal{P}_N}$ by 
\[
\Phi \left( \left( f_j (s) \right)_{j = 1}^r, \left( x_{p_{k_n}} \right)_{n = 1}^N \right) = \left(  \left( \exp ( f_j (s) ) \right)_{j = 1}^r, \left( x_{p_{k_n}} \right)_{n = 1}^N \right).
\]
Since the map $\Phi$ is continuous, 
we have 
\[
\supp\left( \nu_{\underline{\phi},\, \mathcal{P}_N } \right) 
= \overline{ \Phi \left( \supp \left( \log \underline{\phi} (s, \omega), \left( \omega(p) \right)_{p \in \mathcal{P}_N} \right) \right)}
\]
by Lemma B.2.1 in \cite{K2021}.
By Lemma \ref{lem:supplog}, the equtions
\[
\Phi \left( \supp \left( \log \underline{\phi} (s, \omega), \left( \omega(p) \right)_{p \in \mathcal{P}_N} \right) \right)
=\Phi \left( \prod_{j = 1}^{r} H(D_{\phi_j}) \times \mathbb{T}^{\mathcal{P}_N}  \right) 
= \prod_{j = 1}^{r} H (D_{\phi_j})^{\times} \times \mathbb{T}^{\mathcal{P}_N} 
\]
holds, 
where $H(G)^{\times}$ denotes the set of non-vanishing holomorphic function on $G$ for a simply connected region $G$. 
Hence it is enough to show that
\[
\overline{ \prod_{j = 1}^{r} H (D_{\phi_j})^{\times}  \times \mathbb{T}^{\mathcal{P}_N} }
= \prod_{j = 1}^r H_0(D_{\phi_j}) \times \mathbb{T}^{\mathcal{P}_N}.
\]
This equation is proved by Hurwitz's classical result on zeros of uniformly convergent sequence of functions (see e.g. \cite[Lemma A.5.5.]{K2021}).
This completes the proof. 
\end{proof}

\begin{proof}[Proof of Corollary \ref{cor:main}]
Let every settings and assumptions be the same as in Corollary \ref{cor:main}.
By Mergelyan's approximation theorem (see e.g. \cite[Theorem 20.5]{R1987}), 
we can take a non-vanishing polynomial $P_j$ on $K_j$ such that
\begin{equation}\label{eqn:IN1}
\max_{s \in K_j} \left| f_j (s) - P_j (s) \right| < \frac{\epsilon}{3}
\end{equation}
for $j = 1, \ldots, r$.
Since $P_j$ is non-vanishing on $K_j$, 
there exists a simply connected region $G_j$ such that $K_j \subset G_j$ and $P_j$ is non-vanishing on $G_j$ for $j = 1, \ldots, r$. 
Hence we can take a holomorphic branch $\log P_j$ of $P_j$ on $G_j$.
Using Mergelyan's approximation theorem again, 
we can take a polynomial $Q_j$ such that
\begin{equation}\label{eqn:IN2}
\max_{s \in K_j} \left| P_j (s) - \exp( Q_j (s) ) \right| < \frac{\epsilon}{3}
\end{equation}
for $j = 1, \ldots, r$, 
where we use the inequality
\[
\max_{s \in K_j} \left| P_j (s) - \exp( Q_j (s) ) \right|
\leq \max_{s \in K_j} \left| P_j(s) \right| \cdot \max_{s \in K_j} \left| 1 - \exp \left( Q_j (s) - \log P_j (s) \right) \right|.
\]
Put
\[
U_j = \left\{ g \in H \left( D_\phi \right) \mid \max_{s \in K_j} \left| g(s) - \exp \left( Q_j (s) \right) \right| < \frac{\epsilon}{3} \right\}
\]
for $j = 1, \ldots, r$.
Note that $U_j$ is an open neighborhood of $\exp \left( Q_j(s) \right)$.

Next, let
\begin{align*}
V_{p_{k_n}} 
&= \left\{ z \in \mathbb{T}_{p_{k_n}} \mid d_{p_{k_n}} \left( z, \exp \left( 2 \pi i \theta_{p_{k_n}} \right) \right) < \epsilon \right\} \\
&= \left\{ \exp(2 \pi i \theta) \in \mathbb{T}_{p_{k_n}} \mid \left\| \theta - \theta_{p_{k_n}} \right\| < \epsilon \right\}
\end{align*}
for $n = 1, \ldots, N$.
Note also that $U_{p_{k_n}}$ is an open neighborhood of $\exp \left( 2 \pi i \theta_{p_{k_n}} \right)$ and that the condition $p_{k_n}^{i \tau} \in V_{k_n}$ is equivalent to $\left\| \tau \frac{\log p_{k_n}}{2 \pi} - \theta_{p_{k_n}} \right\| < \epsilon$. 
Put
\[
\bm{A} 
= \left( U_1 \times \cdots \times U_r \right) \times \left( V_{p_{k_1}} \times \cdots \times V_{p_{k_N}} \right) \subset \prod_{j = 1}^r H(D_{\phi_j}) \times \mathbb{T}^{\mathcal{P}_N}. 
\]
Then, by the portmanteau theorem (see e.g. \cite[Theorem 2.1]{B1999}), the inequality \eqref{eqn:IN1}, the inequality \eqref{eqn:IN2} and Theorem \ref{thm:main} (ii), 
we have
\def\QQ{\quad \quad \quad \quad \quad \quad \quad \quad \quad \quad \quad \quad \quad \quad \quad \quad \quad \quad \quad}
\begin{align*}
& \liminf_{T \rightarrow \infty} \frac{1}{T} \meas \Bigg{\{} \tau \in [0, T] ~;~ \max_{1 \leq j \leq n} \max_{s \in K_j}  \left|  \phi_j(s + i \tau) - f_j(s) \right| < \epsilon,\\ 
& \QQ \max_{1 \leq n \leq N} \left\| \tau \frac{\log p_{k_n}}{2 \pi} - \theta_{p_{k_n}} \right\| < \epsilon \Bigg{\}} \\ 
\geq&\, \liminf_{T \rightarrow \infty} \frac{1}{T} \meas \Bigg{\{} \tau \in [0, T] ~;~ \max_{1 \leq j \leq n} \max_{s \in K_j}  \left|  \phi_j(s + i \tau) - \exp\left( Q_j (s) \right) \right| < \frac{\epsilon}{3},\\ 
& \QQ \max_{1 \leq n \leq N} \left\| \tau \frac{\log p_{k_n}}{2 \pi} - \theta_{p_{k_n}} \right\| < \epsilon \Bigg{\}} \\
=&\, \liminf_{T \rightarrow \infty} \nu_{T,\, \underline{\phi},\, \mathcal{P}_N } \left( \bm{A} \right)
\geq \nu_{\underline{\phi},\, \mathcal{P}_N } \left( \bm{A} \right)
>0.
\end{align*}
This completes the proof.
\end{proof}

\begin{center}
\textbf{Acknowledgements.}
\end{center}

The author would like to express his gratitude to Professor Kohji Matsumoto for reading the draft of the paper and providing valuable comments. 
The author also would like to thank Professor Takashi Nakamura for pointing out a minor mistake at RIMS Workshop 2024 on Analytic Number Theory and Related Topics.

\begin{flushleft}
{\footnotesize
{\sc
Department of Electronic and Information Engineering, \\
National Institute of Technology, Suzuka College,\\
Shiroko-cho, Suzuka, Mie, 510-0294, Japan
}\\
{\it E-mail address}: {\tt endo-k@info.suzuka-ct.ac.jp}
}
\end{flushleft}

\end{document}